\documentclass[20pt,reqno]{amsart}  %%%% new
\usepackage{amsfonts}
\usepackage{graphicx}
\usepackage{amscd}
\usepackage{color}
\usepackage{amsmath}
\usepackage{amsfonts,amssymb,amsthm,mathrsfs,xypic,cite,leftidx,enumitem,comment}
\xyoption{curve}
\usepackage{fancybox}
\usepackage{tikz}
\usepackage{tikz-cd}
\usetikzlibrary{arrows,decorations.pathmorphing,backgrounds,positioning,fit,petri}
\usetikzlibrary{arrows,patterns,matrix}
\usetikzlibrary{decorations.markings}
\usepackage[all,cmtip]{xy}
\usepackage{caption}
\usepackage{subcaption}

\newcommand{\m}{\Lambda}
\newcommand{\Hom}{\operatorname{Hom}}
\newcommand{\End}{\operatorname{End}}
\newcommand{\Aut}{\operatorname{Aut}}
\newcommand{\Ker}{\operatorname{Ker}}
\newcommand{\cok}{\operatorname{Coker}}
\newcommand{\Ima}{\operatorname{Im}}
\newcommand{\rad}{\operatorname{rad}}
\newcommand{\Ext}{\operatorname{Ext}}

\newcommand{\op}{\operatorname{op}}

\newcommand{\Gp}{{\rm Gproj}}

\newcommand{\smon}{{\rm smon}}

\newcommand{\modd}{\textendash\operatorname{mod}}
\newcommand{\pmodd}{\textendash\underline{\operatorname{mod}}}

\newcommand{\Mod}{\textendash\operatorname{Mod}}

\newcommand{\Mimo}{\operatorname{\bf Mimo}}
\newcommand{\Mepi}{\operatorname{\bf Mepi}}
\newcommand{\mimo}{\operatorname{\rm mimo}}
\newcommand{\mepi}{\operatorname{\rm mepi}}
\newcommand{\bKer}{\operatorname{\bf Ker}}
\newcommand{\bcok}{\operatorname{\bf Cok}}
\newcommand{\Rot}{\operatorname{\bf Rot}}
\newcommand{\Htp}{{\rm Htp}}
\newcommand{\Tr}{\operatorname{Tr}}

\newcommand{\calA}{{\mathcal A}}

\newcommand{\calC}{{\mathcal C}}
\newcommand{\calD}{{\mathcal D}}
\newcommand{\calE}{{\mathcal E}}
\newcommand{\calF}{{\mathcal F}}

\newcommand{\calH}{{\mathcal H}}
\newcommand{\calI}{{\mathcal I}}

\newcommand{\cP}{{\mathcal P}}

\newcommand{\calS}{{\mathcal S}}
\newcommand{\calT}{{\mathcal T}}

\newcommand{\calV}{{\mathcal V}}

\newcommand{\bS}{{\mathbb S}}
\newcommand{\rt}{\rightarrow}

\newcommand{\la}{\langle}
\newcommand{\ra}{\rangle}
\newcommand{\st}{\stackrel}

\newtheorem{thm}{Theorem}[section]

\newtheorem{cor}[thm]{Corollary}

\newtheorem{propdef}[thm]{Proposition-Definition}
\newtheorem{lem}[thm]{Lemma}
\newtheorem{exm}[thm]{Example}

\newtheorem{prop}[thm]{Proposition}
\newtheorem{rem}[thm]{Remark}
\newtheorem{defn}[thm]{Definition}
\newtheorem{constr}[thm]{Construction}
\begin{document}

\title [AR translation in the monomorphism categories of exact categories]
{Auslander-Reiten translations in the monomorphism categories of exact categories}
\author [Xiu-Hua Luo, Shijie Zhu ] {Xiu-Hua Luo, Shijie Zhu$^*$}
\thanks{$^*$The corresponding author.}

\thanks{Supported by the NSF of China (No. 12201321).}
\thanks{xiuhualuo$\symbol{64}$ntu.edu.cn  \ \ \ \ shijiezhu$\symbol{64}$ntu.edu.cn}

 \maketitle

\begin{center}
School of Mathematics and Statistics, \ \ Nantong University\\
Jiangsu 226019, P. R. China
\end{center}

\begin{abstract} Let $\calV$ be a dualizing $k$-variety over a field $k$ and $\calC$ be a functorially finite exact subcategory of $\calV\modd$ having enough $\Ext$-projective and $\Ext$-injective objects. Let  $\mathcal H (\mathcal C)$ (resp. $\mathcal S (\mathcal C)$ and $\mathcal F (\mathcal C)$) be its morphism (resp. monomorphism and epimorphism) category. It turns out that these categories have almost split sequences. We show an explicit formula for the Auslander-Reiten translation in $\mathcal S (\mathcal C)$. Furthermore, if $\mathcal C$ is a stably $d$-Calabi-Yau Frobenius category, we calculate objects under powers of Auslander-Reiten translation in the triangulated category $\overline{\mathcal S(\mathcal C)}$.
\vskip5pt

{\it Key words and phrases.  \ (mono)morphism categories, \ exact categories, \ Auslander-Reiten translation, \ Auslander-Reiten-Serre duality, \ functor categories.}

{\it MSC2020: 16G70, 18A20, 18A25, 16B50.}
\end{abstract}

\section {Introduction}

Throughout the paper, $k$ is a field and $\Lambda$ is a finite dimensional algebra over $k$. Denote by $D=\Hom_k(-,k)$ the $k$-dual functor.

\subsection{Background}

The classification of subobjects is a fundamental problem in mathematics. In  \cite{Birk}, G. Birkhoff proposed the problem of classifying all
subgroups of abelian $p$-groups by matrices. This question is related to representations of partial order sets, such as \cite{Ar,Si1,Si2}. C. M. Ringel and M. Schmidmeier \cite{RS1,RS2,RS3 } revisited this subject using a modern point of view by studying the representation type and the Auslander-Reiten theory of the submodule category for artin algebras $k[x]/(x^n)$, for $n>1$. In \cite{RS2}, they developed the Auslander-Reiten theory in the submodule category for an artin algebra and showed that the Auslander-Reiten translation in the submodule category $\calS(\Lambda)$ can be computed by the Auslander-Reiten translation within the module category.
\begin{thm}\label{thm:RS}
Let $\Lambda$ be an artin algebra. Then the submodule category $\calS(\m)$ has almost split sequences and the Auslander-Reiten translation in $\calS(\m)$ satisfies
\begin{align}
\tau_\calS f= \Mimo\tau_\Lambda \cok f
\end{align}
where $\Mimo$ denotes the minimal monomorphism. 
\end{thm}

This paper is devoted to study the morphism categories $\calH(\calC)$, respectively monomorphism categories $\calS(\calC)$ and epimorphism categories $\calF(\calC)$ of exact categories $\calC$, where the objects are morphism, respectively inflations and deflations in $\calC$. It turns out that there have been fruitful applications from this point of view. In \cite{LiZ}, the authors showed the following:
\begin{thm}\label{thm:t2a}
    For a Gorenstein algebra $\Lambda$, the category of Gorenstein projective modules over the matrix algebra $T_2(\Lambda)=\begin{bmatrix}
    \Lambda &\Lambda \\0&\Lambda 
\end{bmatrix}$ is $\Gp(T_2(\Lambda))\cong \calS(\Gp(\Lambda))$.
\end{thm} 

In this case, it is well-known that $\calS(\Gp(\Lambda))$ is a functorially finite subcategory of $T_2(\Lambda)\modd$, and hence it has almost split sequences.
In \cite{LZ3}, the authors   showed a similar result about the existence of almost split sequences in the monomorphism categories of complete hereditary cotorsion classes. 
\begin{thm}
If $\calC$ is a contravariantly finite resolving subcategory in $\m\modd$, then so is $\calS(\calC)$ in $T_2(\Lambda)\modd$ and hence it has almost split sequences.
\end{thm}

\subsection{Main results}
 Motivated by these  results above, we will focus on studying the existence of almost split sequences in the monomorphism category of an exact category  $\calC$ and their Auslander-Reiten translations in this paper. We will show an explicit formula for the Auslander-Reiten translation in the monomorphism category $\calS (\calC)$ and the epimorphism category $\calF(\calC)$.

\begingroup %thm A B C
\setcounter{thm}{0}
\renewcommand\thethm{\Alph{thm}}

\begin{thm}\label{thm:A}
Let $\calV$ be a dualizing $k$-variety and $\calC$ be a functorially finite exact subcategory of $\calV\modd$ having enough   $\Ext$-projective and $\Ext$-injective objects. Then the followings hold:
\begin{enumerate}
\item The monomorphism category $\calS(\calC)$ and the epimorphism category $\calF(\calC)$ have almost split sequences.
\item Let $\tau_\calC$ denote the Auslander-Reiten translation of $\calC$, then the Auslander-Reiten translations of $\calS(\calC)$ and $\calF(\calC)$  %and their inverses
are given by:
$$
\tau_\calS(f)\cong\Mimo\tau_\calC \underline{\bcok(f)}, \ \ \tau_\calF(f)\cong\bcok\Mimo\tau_\calC(\underline f),$$
where  $\underline{\bcok (f)}$ and $\underline f$ are the images of the maps in the stable category $\underline \calC$.
%$$
%\tau_\calS^{-1}\cong\Ker\Mepi\tau_\calC^{-1}, \ \ \tau_\calF^{-1}\cong\Mepi\tau_\calC^{-1}\Ker.
%$$
\end{enumerate}
\end{thm}

We need to clarify the notations as well as make some remarks concerning this result.

\begin{enumerate}[label=(\roman*)]
 \item {\bf Exact structures and almost split sequences.} When $\calC$ is an exact category, the morphism category $\calH(\calC)$ and its exact subcategories $\calS(\calC)$, $\calF(\calC)$ have a degreewise exact structure inherited from $\calC$ (see Definition \ref{defn:deg_exact}). The almost split structure in Theorem \ref{thm:A} is considered with respect to this exact structure. 

 Some special almost split sequences in $\calH(\calC)$, $\calS(\calC)$ and $\calF(\calC)$ are classified in Lemma \ref{lem:sporadic}. The others are so-called cw-exact i.e. degreewise split exact (see Definition \ref{defn:cw_exact}).

\item {\bf RSS-equivalence.} Let $A\stackrel{f}\to B\stackrel{g}\to C$ be a conflation in an exact category $\calC$. We denote by $\cok f=g$ and $\Ker g=f$. In fact, this induces a pair of equivalence functors $\bcok:\calS(\calC)\to \calF(\calC)$ and $\bKer:\calF(\calC)\to\calS(\calC)$, which are called the RSS-equivalences. Furthermore, if $\calC$ has enough   $\Ext$-projective objects, they induce functors on the stable categories $\bcok:\underline{\calS(\calC)}\to \underline {\calF(\calC)}$ and $\bKer:\underline{\calF(\calC)}\to \underline{\calS(\calC)}$.

\item {\bf $\Mimo$ and epivalence.} The notion of minimal monomorphism was introduced in \cite{RS2}. For any morphism $f:A\to B$ in a module category $\m\modd$, the minimal monomorphism $\Mimo f=[f,e]^T:A\to B\oplus I\Ker f$, where $e$ is an extension of the injective envelope $i:\Ker f\to I\Ker f$ by the inclusion $\Ker f\hookrightarrow A$. It was shown that $\Mimo f$ is unique up to isomorphism, independent on the choice of $e$ and it is a minimal right $\calS(\m)$-approximation of $f$. In \cite{BM}, the authors pointed out that $\Mimo$ is a functor $\calH(\m)\to \overline{\calS(\m)}$. In \cite{LS,GKKP2}, it was shown that right $\calS(\m)$-approximations $\mathfrak R f\to f$  give rise to a functor $\mathfrak R:\calH(\m)\to \overline{\calS(\m)}$. In this paper, we generalize this notion for a Krull-Schmidt exact category $\calC$ with enough  $\Ext$-injective objects and denote the image of right $\calS(\calC)$-approximations of $f$ in the stable subcategory $\overline{\calS(\calC)}$ by $\Mimo f$. Hence, $\Mimo:\calH(\calC)\to\overline{\calS(\calC)}$ is a functor.

Another important functor is the epivalence $\pi:\overline{\calS(\calC)}\to \calH(\overline \calC)$ (see Definition \ref{def:epiv}, Theorem \ref{thm:epivalence}), which sends an inflation $f$ to its image $\overline f$ in the stable category $\overline\calC$ and $\pi':\underline{\calF(\calC)}\to \calH(\underline \calC)$ which sends a deflation $f$ to its image $\underline f$ in $\underline \calC$. Since an epivalence induces a bijection of indecomposable objects between categories, it turns out that the inverse image of an indecomposable object $\overline f$ in $\calH(\overline \calC)$ under the functor $\pi$ agrees with $\Mimo f$. Hence we denote $\Mimo \overline f:=\Mimo f$ as an object in $\overline{\calS(\calC)}$ (Remark \ref{rem:mimo_well_def}).

\item {\bf Functoriality.} We clarify the notations in these two Auslander-Reiten formulas.
In general, if a Krull-Schmidt exact category $\calC$ with enough  $\Ext$-projective and $\Ext$-injective objects has almost split sequences, then the Auslander-Reiten translation is a functor $\tau_\calC:\underline\calC\to \overline\calC$. However, our formulas for $\tau_\calS$ and $\tau_\calF$ only hold for objects $f$ in $\underline{\calS(\calC)}$ or $\underline{\calF(\calC)}$, but not for morphisms $(\alpha,\beta):f\to g$ in $\underline{\calS(\calC)}$ or $\underline{\calF(\calC)}$. 
Well-definedness of the formulas can be checked directly from the following compositions:
$$
\begin{tikzcd}
\tau_\calS:\underline{\calS(\calC)}\ar[r,"\bcok"]& \underline{\calF(\calC)} \ar[r,"\pi'"]& \calH(\underline\calC) \ar[r,"\tau_\calC"]&\calH(\overline \calC)\ar[r,"\Mimo"]& \overline{\calS(\calC)}
%\\
%f\ar[r]&\bcok f \ar[r]&\underline{\bcok f} \ar[r]& \tau\underline{\bcok f} \ar[r]& \Mimo\tau\underline{\bcok f}
\end{tikzcd}
$$$$
\begin{tikzcd}
\tau_\calF:\underline{\calF(\calC)} \ar[r,"\pi'"]& \calH(\underline\calC)\ar[r, "\tau_\calC"]& \calH(\overline\calC) \ar[r,"\Mimo"] & \overline{\calS(\calC)}\ar[r,"\bcok"] &
\overline{\calF(\calC)}
\end{tikzcd}
$$

\item {\bf Comparison of the proofs.} Although the formulas seem quite similar than those in Theorem \ref{thm:RS}, we would like to mention that we have to choose a different route of proof in this paper. The original proof in \cite{RS2} used techniques of transferring almost split sequences. Since there are functors $\bKer:\calH(\m)\to \calS(\m)$ and 
$\bcok:\calH(\m)\to \calF(\m)$ which preserve certain almost split sequences, one can transfer the problem of calculating the Auslander-Reiten translation in $\calS(\m)$ and $\calF(\m)$ to the problem of calculating Auslander-Reiten translation in $\calH(\m)$, which is equivalent to the module category $T_2(\Lambda)\modd$. Then the Auslander-Reiten translation in $\calH(\m)$ can be computed as $\tau_\calH=D\Tr$, where $\Tr$ is the transpose functor on $T_2(\Lambda)\modd$. However, one may immediately realize that substituting $\calH(\Lambda)$ by $\calH(\calC)$ will cause the problem that there no longer exists  a ``Nakayama functor'' or a ``transpose functor'' on $\calH(\calC)$, since it is not necessarily a module category. 

To solve this problem, we need to transfer almost split sequences through functor categories. The key observation is due to a recent survey by Hafezi and Eshraghi \cite{HE}, where they studied the functor $\Theta: \calH(\calC)\to \calC\modd$ from the morphism category to the category of finitely presented functors on an additive category $\calC$, showing that the functor $\Theta$ preserves cw-exact almost split sequences of $\calH(\calC)$ (see Theorem \ref{thm:EH}).

We briefly explain the strategy of our proof here. First of all, up to some reductions it  suffices to prove the formula $\tau_\calF(f)\cong\bcok\Mimo\tau_\calC(\underline f)$ holds for deflations $f$ which is the end term of a cw-exact almost split sequence (see Section \ref{subsec:proof}). We use the stable functor $\calF(\calC)\to \calH(\underline \calC)$ ($f\mapsto \underline f$) to transfer cw-exact almost split sequences in $\calF(\calC)$ to cw-exact almost split sequences in $\calH(\underline \calC)$ (Lemma \ref{lem:stable almost split}). And then we can  use the functor $\Theta:\calH(\underline\calC)\to \underline\calC\modd$ to transfer cw-exact almost split sequences in $\calH(\underline \calC)$ to almost split sequences in $\underline\calC\modd$. Finally, the problem boils down to computing the Auslander-Reiten translation on $\underline\calC\modd$, which has been investigated by \cite{AR74}.

\item {\bf Generalizations and relevant results.} In this paper, we focus on the ``submodule categories'' as its importance of being a prototype of the general (separated) monomorphism categories which have been studied intensively in \cite{LZ1,LZ2,ZX}. Meanwhile, further investigations and generalizations are expected. Possible generalizations are two-folded. 
On one hand, in contrast of our assumption that $\calC$ is an exact subcategory of an abelian category and has enough $\Ext$-projective and $\Ext$-injective objects, one may consider the Auslander-Reiten theory of $\calS(\calC)$ for a $\Hom$-finite Krull-Schmidt exact category $\calC$ having almost split sequences, but without enough $\Ext$-projective or $\Ext$-injective objects (e.g. the category of coherent sheaves ${\rm coh} \mathbb P^1$). For such categories $\calC$, it is beyond our knowledge so far to show that $\calS(\calC)$ also has almost split sequences. 
On the other hand, one may consider the Auslander-Reiten theory for separated monomorphism categories $\smon(Q,\calC)$ of an exact category $\calC$ over an acyclic quiver $Q$. In \cite{XZZ}, the Auslander-Reiten translation formula was given for an $\mathbb A_n$ type quiver $Q$ with linear orientation and $\calC=\m\modd$ for some artin algebra $\m$. In \cite{GKKP3}, the authors explicitly described the Auslander-Reiten translation for the monomorphism category of dualisable pro-species. 
\end{enumerate}

Especially, we are interested in the monomorphism categories $\calS(\calC)$ of Frobenius categories $\calC$, which by \cite{C1} are also Frobenius. 
 As an application, when $\calC$ is a stably $d$-Calabi-Yau Frobenius category, we calculate objects under the powers of Auslander-Reiten translation in the triangulated category $\overline{\calS(\calC)}$.
\vskip5pt
 
\begin{thm}\label{thm:B}
Let $\calV$ be a dualizing $k$-variety and $\calC$ be a functorially finite stably $d$-Calabi-Yau Frobenius category of $\calV\modd$. Let $\la 1\ra$ be the suspension functor of the stable category $\overline{\calS(\calC)}$. The Auslander-Reiten translation $\tau_\calS:\overline{\calS(\calC)}\to\overline{\calS(\calC)}$ satisfies 
$
\tau_\calS^6 f\cong f\la 6d-4\ra 
$
for every object $f$ in $\overline{\calS(\calC)}$.
\end{thm}
\endgroup

\subsection{Outline of the paper}
This paper is organized in the following way. In section 2, we recall the definitions of exact categories, stable categories and almost split sequences. In section 3, we recall the exact structures on morphism categories and characterize almost split sequences in morphism categories as well as their relation with almost split sequences in the functor categories. In section 4, we study the existence of almost split sequences in monomorphism categories and prove Theorem \ref{thm:A}(1). Section 5 is devoted to studying the minimal monomorphisms and epivalences. In section 6, we present the proof of Theorem \ref{thm:A}(2). In section 7,  Auslander-Reiten translations of monomorphism categories of Frobenius categories are investigated. In section 8, we provide examples to further illustrate our results.

\section{Almost split sequences in exact categories}
 \subsection {Exact categories and (co)stable categories}
First, we recall some fundamental notions of exact categories in the sense of Quillen \cite{Q} and Keller \cite[Appendix A]{Ke}. One can also refer to \cite{B} for a detailed survey.

An exact structure $\calE$ on an additive category $\calC$ is a fixed class of composable morphism pairs $A\stackrel{i}\to B\stackrel{p}\to C$, where $i$ is a kernel of $p$ and $p$ is a cokernel of $i$, satisfying certain axioms, see e.g. \cite[Definition 2.1]{B}. An {\bf exact category} $(\calC,\calE)$ is a pair of an additive category $\calC$ and an exact structure $\calE$ on it. An element  $A\stackrel{i}\to B\stackrel{p}\to C$ of $\calE$ is called a conflation or short exact sequence, where $i$ is called an inflation or admissible monomorphism and $p$ is called a deflation or admissible epimorphism. Later on, when there is no confusion, we simply call $\calC$ an exact category, omitting the exact structure $\calE$. A conflation is usually denoted by $0\to A\stackrel{i}\to B\stackrel{p}\to C\to 0$.

\begin{exm}\label{exm:e_split}
For any additive category $\calC$, denote by $\calE^{\rm split}$ the exact structure containing only split exact sequences in $\calC$. Then $\calC^{\rm split}:=(\calC,\calE^{\rm split})$ is an exact category. In fact, $\calE^{\rm split}$ is the smallest exact structure on $\calC$.     
\end{exm}

Let $(\calC,\calE)$ and $(\calC', \calE')$ be exact categories. We say an additive functor $F:\calC\to \calC'$ preserves exact sequences if $FA\st{Ff}\to FB\st{Fg}\to FC \in \calE'$ whenever $A\st{f}\to B\st{g}\to C\in\calE.$ We say $F$ reflects exact sequences, if $A\st{f}\to B\st{g}\to C\in\calE$ whenever $FA\st{Ff}\to FB\st{Fg}\to FC \in \calE'$. An exact functor is an additive functor which preserves exact sequences. An {\bf exact equivalence} is an additive equivalence functor $F:(\calC,\calE)\to (\calC',\calE')$, which preserves and reflects exact sequences.

\begin{exm}
 Let $(\calC,\calE)$ be an exact category and $\calE$ contain at least one non-split exact sequences. Then the identity functor $1:(\calC,\calE^{\rm split})\to (\calC,\calE)$ is an exact functor, but not an exact equivalence.   
\end{exm}

Let $(\calC,\calE)$ be an exact category. A full subcategory $\calD\subseteq \calC$ is called an {\bf exact subcategory} of $(\calC,\calE)$ if $\calD$ is closed under extensions and equipped with an exact structure which contains $A\to B\to C\in\calE$, where $A,B,C\in\calD$. In some literature, exact subcategory are called fully exact subcategory.

In an exact category $(\calC,\calE)$, conflations $0\to A\to B\to C\to 0$ form an abelian group $\Ext^1(C,A)$ under Baer sum, which is also an $\End(A)$-$\End(C)$-bimodule with the actions given by pushouts and pullbakcs respectively. A morphism $f:X\to Y$ is called projectively trivial if the induced map $\Ext^1(f,M):\Ext^1(Y,M)\to \Ext^1(X,M)$ is zero for any object $M\in\calC$. An object $P$ is called $\Ext$-projective if the identity $1_P$ is projectively trivial. We say $\calC$ has enough $\Ext$-projective objects if for each $X\in\calC$, there is a deflation $f:P\to X$ for some $\Ext$-projective object $P$. Notions with respect to injectively trivial morphisms and $\Ext$-injective objects can be defined dually.

\begin{rem}
When $\calC$ has enough $\Ext$-projective (injective) objects, a morphism $f$ is projectively (injectively) trivial if and only if it factors through an $\Ext$-projective (injective) object. When $\calC$ is an abelian category and $\calE$ consists of all the short exact sequences in $\calC$, then $\Ext$-projective (injective) objects are simply projective and injective objects in $\calC$.  
\end{rem}

Let $\calC$ be an exact category. Denote by $\underline \calC$ the stable category, whose objects are the same as $\calC$ and morphism sets are $\Hom_{\underline \calC}(X,Y)=\Hom_{\calC}(X,Y)/\cP(X,Y)$, where $\cP(X,Y)$ denotes the set of projectively trivial morphisms from $X$ to $Y$. The costable category $\overline\calC$ is defined dually. Given a morphism $f\in\calC$, denote by $\underline f$ (resp. $\overline f$) the class of morphisms represented by $f$ in $\underline \calC$ (resp. $\overline\calC$). Define $Q_\calC:\calC\to \underline\calC$ the stable functor which sends morphisms $f\mapsto \underline f$. The costable functor $Q^\calC: \calC\to \overline \calC$ is defined dually. A {\bf Frobenius category} is an exact category $(\calC,\calE)$ with enough $\Ext$-projectives and $\Ext$-injectives, moreover, the $\Ext$-projectives coincide with the $\Ext$-injectives.

\begin{thm}\cite{Hap1}
Let $\calC$ be a Frobenius category. Then $\underline \calC=\overline \calC$ is a triangulated category.    
\end{thm}

\subsection{Krull-Schmidt categories.} We are mostly interested in additive categories which are Krull-Schmidt. 
Recall that an additive category $\mathcal C$ is called {\bf Krull-Schmidt} if any object $M\in\mathcal C$ has a unique (up to permutation) decomposition: $M=\mathop{\oplus}\limits_{i=1}^{n}M_i$,where $\End_{\mathcal C}(M_i)$ is a local ring for all $1\leq i\leq n$.    
It is well-known that a category is Krull-Schmidt if and only if it has split idempotents and the endomorphism ring of each object is semi-perfect \cite[Corollary 4.4]{Kr}. In particular, $\Hom$-finite idempotent split $k$-linear categories are Krull-Schmidt.

Recall that in an additive category $\calC$, two morphisms $f:X\to Y$ and $f':X'\to Y$ are called right equivalent if there exist $g:X\to X'$ and $h:X'\to X$ such that $f=f'g$ and $f'=fh$. A morphism $f:X\to Y$ is called right minimal if any endomorphism $g\in \End(X)$ such that $f=fg$ is an automorphism.  
It is well-known that in a Krull-Schmidt category, any morphism $f:X\to Y$ has a unique up to right equivalence decomposition $f=(f_1,f_2):X=X_1\oplus X_2\to Y$ such that $f_1$ is right minimal and $f_2=0$ (see \cite[Corollary 1.4]{KS}), where $f_1$ is called the right minimal version of $f$. The notion of left minimal mophisms is defined dually.

A morphism $f:P\to M$ in an exact category $\mathcal C$ is called an {\bf $\Ext$-projective cover}, if $f$ is a right minimal deflation and $P$ is $\Ext$-projective. The dual notion of {\bf $\Ext$-injective envelopes} are defined similarly. Due to the existence of right (left) minimal morphisms in Krull-Schmidt categories, it is easy to see the existence of $\Ext$-projective cover and $\Ext$-injective envelopes. 

\begin{prop}\label{prop:KS_cover}
    Let $\calC$ be a Krull-Schmidt exact category with enough $\Ext$-projective (resp. $\Ext$-injective) objects. Then each object admits an $\Ext$-projective cover (resp. $\Ext$-injective envelope).
\end{prop}

Denote by $\calC_\cP$ (resp. $\calC_\calI$) the full subcategory consisting of objects in $\calC$ which do not contain any non-zero $\Ext$-projective (resp. $\Ext$-injective) summands. There is an equivalence $\underline{\calC_\cP}\cong \underline\calC$ (resp. $\overline{\calC_\calI}\cong \overline\calC$)  induced by the inclusion $\calC_\cP\to\calC$ (resp. $\calC_\calI\to\calC$). The following result is easy to prove (cf.\cite[IV. Exercise 7]{ARS}).

\begin{lem}\label{lem:cp_iso}
Let $\calC$ be a  Krull-Schmidt exact category with enough $\Ext$-projectives. A morphism $f:A\to B$ in $\calC_\cP$ is an isomorphism if and only if $\underline f$ is an isomorphism in $\underline \calC$. 
\end{lem}

 We introduce the notion of epivalences following \cite[Definition 5.5]{GKKP2}.  
 \begin{defn}\label{def:epiv}
    An additive functor $F:\calC\to \calD$ is said to be an {\bf epivalence}, if it is full, dense and reflects isomorphisms.
 \end{defn} 
 
 An important fact about epivalences is that $F$ preserves and reflects indecomposable objects. Hence an epivalence $F:\calC\to \calD$ immediately yields a bijection between indecomposable objects in $\calC$ and $\calD$. Notice that the notion of epivalences was introduced in \cite[Chapter II]{Aus} by the name of ``representation equivalence''. We prefer to use the name ``epivalence'', because it emphasizes the fact that for an epivalence $F:\calC\to \calD$, $\calC$ may contain more morphisms.
 Now we show an example of epivalence. 
\begin{lem}\label{lem:cp_epiv}
Let $\calC$ be a Krull-Schmidt exact category with enough  $\Ext$-projectives. Let $\iota$ be the inclusion of $\calC_\cP$ as a full subcategory of $\calC$. The restriction $Q_{\calC}|_{\calC_{\cP}}=Q_\calC\circ \iota:\calC_\cP\to \underline\calC$ is an epivalence.
\end{lem}
\begin{proof}
It is easy to see  $Q_\calC\circ\iota$ is full and dense. It reflects isomorphisms due to Lemma \ref{lem:cp_iso}.
\end{proof}

\begin{cor}\label{cor:qc_epiv}
The functor $Q_\calC|_{\calC_{\cP}}$ induces a bijection between indecomposable objects in $\calC_\cP$ and~$\underline\calC$.
\end{cor}

\subsection{Almost split sequences}
Following \cite[Definition 3.4]{INP}, we introduce the Auslander-Reiten-Serre duality for exact categories.  

\begin{defn}\label{defn:ARS_duality}
 Let $\calC$ be a $k$-linear exact category. \\
 (1) A right Auslander-Reiten-Serre (ARS) duality is a pair $(\tau, \eta)$ of an additive functor $\tau:\underline\calC\to\overline\calC$ and a binatural isomorphism
$\eta_{A,B} : \Hom_{\underline\calC}(A, B) \cong D\Ext^1_\calC(B, \tau A)$ for any $A, B \in\calC$ .\\
(2) If moreover $\tau$ is an equivalence, we say that $(\tau,\eta)$ is an Auslander-Reiten-Serre (ARS) duality.
\end{defn}

\begin{defn}\label{defn:ass_ex_cat}
Let $\calC$ be an additive category. We call $f:A\to B$ in $\calC$ a source map or left almost split, if it is not a section and each $u: A\to B'$ which is not a section factors through $f$. Dually, we call $g:B\to C$ a sink map or right almost split, if it is not a retraction and each $v: B'\to C$ which is not a retraction factors through $g$.
Furthermore, if $\calC$ is an exact category. An almost split sequence
is a conflation $A\stackrel{f}\to B\stackrel{g}\to C$ such that $f$ is minimal left almost split and $g$ is minimal right almost split. 
  We say an exact category $\calC$ has {\bf left (right) almost split sequences} if every endo-local non-$\Ext$-injective (projective) object is the starting (ending) term of an almost split sequence. 
\end{defn}

 Notice that the definitions of source maps and sink maps are based on the homomorphisms of the category $\calC$ and irrelevant to the exact structures over $\calC$. However, we want to emphasize that almost split sequences depend not only on the category but also on the exact structure.  

\begin{lem}\label{exact-AR} Let $\calE\subseteq\calE'$ be two exact structures on $\calC$.  If $\delta: 0\to A \to B\to C\to 0$ is an almost split sequence in $\calE$, then it is still almost split in $\calE'$.
\end{lem}

\section{Morphism categories and functor categories}

Let $\calC$ be an additive category. The objects in the morphism category $\calH(\calC)$ are morphisms \mbox{ $f:X_1\to X_2$} in $\mathcal C$ also denoted by $\left(\begin{smallmatrix} X_1\\ X_2\end{smallmatrix}\right)_f$. A morphism in $\calH(\calC)$
$$
\left(\begin{smallmatrix} \phi_1 \\ \phi_2 \end{smallmatrix}\right): \left(\begin{smallmatrix} X_1\\ X_2\end{smallmatrix}\right)_{f} \to {\left(\begin{smallmatrix}Y_1\\ Y_2\end{smallmatrix}\right)}_{g}
$$
is a pair of $\calC$-morphisms $\phi_1$, $\phi_2$ such that $\phi_2 f=g\phi_1$. 

\subsection{Exact structures on morphism categories} 

\begin{defn}[The {\rm deg}-exact structure]\label{defn:deg_exact} Let $(\calC,\calE)$ be an exact category.
 Let  $\calE^{\rm deg}$ be the class of all pairs of composable morphisms
$$\xymatrix@1{ & {\left(\begin{smallmatrix} X_1\\ X_2\end{smallmatrix}\right)}_{f}
		\ar[rr]^-{\left(\begin{smallmatrix} \phi_1 \\ \phi_2 \end{smallmatrix}\right)}
		& & {\left(\begin{smallmatrix}Y_1\\ Y_2\end{smallmatrix}\right)}_{g}\ar[rr]^-{\left(\begin{smallmatrix} \psi_1 \\ \psi_2\end{smallmatrix}\right)}& &{\left(\begin{smallmatrix}Z_1 \\ Z_2\end{smallmatrix}\right)}_{h}&  }$$
in $\calH(\calC)$ such that the induced composable morphisms $X_i\st{\phi_i}\rt Y_i\st{\psi_i}\rt Z_i$ are in $\calE$ for $i=1, 2$.
\end{defn}

It is straightforward to see that $\calE^{\rm deg}$ defines an exact structure on $\calH(\calC)$, which is usually referred to as the degreewise exact structure.
In particular, for any additive category $\calC$ there is an exact structure $\calE^{\rm split}$ as mentioned in Example \ref{exm:e_split}. Then the exact category $(\calC,\calE^{\rm split})$ induces an exact structure on $\calH(\calC)$ which can be explicitly described as below:

\begin{defn}[The {\rm cw}-exact structure] \label{defn:cw_exact}
Let  $\calE^{\rm cw}$ be the class of all pairs of composable morphisms
$$\xymatrix@1{ \delta: & {\left(\begin{smallmatrix} X_1\\ X_2\end{smallmatrix}\right)}_{f}
		\ar[rr]^-{\left(\begin{smallmatrix} \phi_1 \\ \phi_2 \end{smallmatrix}\right)}
		& & {\left(\begin{smallmatrix}Y_1\\ Y_2\end{smallmatrix}\right)}_{g}\ar[rr]^-{\left(\begin{smallmatrix} \psi_1 \\ \psi_2\end{smallmatrix}\right)}& &{\left(\begin{smallmatrix}Z_1 \\ Z_2\end{smallmatrix}\right)}_{h}&  } \ \    $$
such that the induced composable morphisms $X_i\st{\phi_i}\rt Y_i\st{\psi_i}\rt Z_i$ split in $\calC$ for $i=1, 2$.
The elements in $\calE^{\rm cw}$ are called cw-exact sequences. 
\end{defn}

From the definition, {\rm cw}-exact structure arises naturally as a $\rm deg$-exact structure.
\begin{prop}\label{prop:cw_deg}
Let $\calC$ be an additive category.
Then there is an equivalence of exact categories 
$(\calH(\calC),\calE^{\rm cw})\cong(\calH(\calC^{split}),\calE^{\rm deg})$.
\end{prop} 

It is worth mentioning the following easy and useful observation due to \cite{HE}.
\begin{rem}\label{rem:cw_tri_form}
Any {\rm cw}-exact sequence
$$\xymatrix@1{ \delta: & {\left(\begin{smallmatrix} X_1\\ X_2\end{smallmatrix}\right)}_{f}
		\ar[rr]^-{\left(\begin{smallmatrix} \phi_1 \\ \phi_2 \end{smallmatrix}\right)}
		& & {\left(\begin{smallmatrix}Y_1\\ Y_2\end{smallmatrix}\right)}_{g}\ar[rr]^-{\left(\begin{smallmatrix} \psi_1 \\ \psi_2\end{smallmatrix}\right)}& &{\left(\begin{smallmatrix}Z_1 \\ Z_2\end{smallmatrix}\right)}_{h}&  } $$
		is isomorphism to
$$\xymatrix@1{ \delta': & {\left(\begin{smallmatrix} X_1\\ X_2\end{smallmatrix}\right)}_{f}
		\ar[rr]^-{\left(\begin{smallmatrix}  \left[\begin{smallmatrix}1\\0\end{smallmatrix}\right]\\ \left[\begin{smallmatrix}1\\0\end{smallmatrix}\right] \end{smallmatrix}\right)}
		& & {\left(\begin{smallmatrix}X_1\oplus Z_1\\ X_2\oplus Z_2\end{smallmatrix}\right)}_{g'}\ar[rr]^-{\left(\begin{smallmatrix} \left[\begin{smallmatrix}0&1\end{smallmatrix}\right] \\ \left[\begin{smallmatrix}0&1\end{smallmatrix}\right]\end{smallmatrix}\right)}& &{\left(\begin{smallmatrix}Z_1 \\ Z_2\end{smallmatrix}\right)}_{h}&  }  $$
where $g'=\begin{bmatrix} f&q\\0&h\end{bmatrix}$ with $q:Z_1\to X_2$ a possibly non-zero morphism in $\calC$.
\end{rem}

When $\calC$ is an exact category,  there are two exact structures on $\calH(\calC)$, namely $\calE^{\rm cw}$ and $\calE^{\rm deg}$. We denote the exact category $(\calH(\calC), \calE^{\rm deg})$ and $(\calH(\calC), \calE^{\rm cw})$ by ${\calH}(\calC)^{\rm deg}$ and ${\calH}(\calC)^{\rm cw}$ respectively.  
 It is clear that $\calE^{\rm cw}\subseteq\calE^{\rm deg}$. So any {\rm cw}-exact sequence is still exact in $(\calH(\calC),\calE^{\rm deg})$. Intuitively, $\calE^{\rm deg}$ is the more natural exact structure because when $\calC\subseteq \Lambda\modd$ is an exact subcategory, $(\calH(\calC), \calE^{\rm deg})$ is an exact subcategory of $\calH(\Lambda)\cong  T_2(\m)\modd$. Hence, in the following, when we refer to the exact category ${\calH}(\calC)$ without emphasizing the exact structure, we mean $(\calH(\calC),\calE^{\rm deg})$. Although the exact structure $\calE^{\rm cw}$ is much smaller than $\calE^{\rm deg}$ in general, it is surprising that when $(\calH(\calC), \calE^{\rm deg})$ has almost split sequences, most of the almost split sequences are {\rm cw}-exact. Hence, when it comes to characterizing almost split sequences, $\calH(\calC)^{\rm cw}$ becomes a convenient tool.

\subsection{Functors on morphisms categories}\label{sec:fun_morph}
Let $F:\calC\to \calD$ be an additive covariant functor between additive categories $\calC$ and $\calD$. Then $F$ naturally induces an additive functor $\calH(F):\calH(\calC)\to \calH(\calD)$, which sends $\left(\begin{smallmatrix}A\\B\end{smallmatrix}\right)_f$ to 
$\left(\begin{smallmatrix}FA\\FB\end{smallmatrix}\right)_{Ff}$ and sends morphisms $(\alpha,\beta):f\to g$ to $(F\alpha, F\beta):Ff\to Fg$.

\begin{prop}\label{prop:HF_equiv}
(1) Let $F:\calC\to \calD$ be an additive functor. If $F$ is an equivalence, then so is $\calH(F)$. (2) Let $F:\calC\to \calD$ be an additive functor between exact categories.  Then $F$ preserves (reflects) exact sequences if and only if $\calH(F)$ preserves (reflects) exact sequences.
\end{prop}
\begin{proof}
 (1) It suffices to prove $\calH(F)$ is dense and fully-faithful. For the density of $\calH(F)$, let $\left(\begin{smallmatrix}A'_1\\A'_2\end{smallmatrix}\right)_g \in \calH(\calD)$. By the density of $F$, there are $A_1, A_2\in \calC$, such that $F(A_1)\cong A'_1, F(A_2)\cong A'_2$. Since $F$ is full, there is a homomorphism $f: A_1\to A_2$ such that $F(f)\cong g.$ So $\calH(F)$ is dense.
 \vskip5pt
  For the fullness of $\calH(F)$, let $\left(\begin{smallmatrix}A_1\\A_2\end{smallmatrix}\right)_f$ and $\left(\begin{smallmatrix}B_1\\B_2\end{smallmatrix}\right)_g$ be objects in $\calH(\calC)$ and $(\alpha', \beta'):Ff\to Fg$ be a homomorphism in  $\calH(\calD)$, i.e. $Fg\alpha'=\beta'Ff$. By the fullness of $F$, there are homomorphisms $\alpha:A_1\to B_1$ and $\beta:A_2\to B_2$ such that $F\alpha=\alpha'$ and $F\beta=\beta'$. Hence $FgF\alpha=F\beta Ff$. Since $F$ is a faithful functor, we have $g\alpha=\beta f$. That is to say,  $(\alpha,\beta):f\to g$ is a homomophism in $\calH(\calC)$ such that $F(\alpha,\beta)=(\alpha',\beta')$. Hence $\calH(F)$ is full. It is easy to see that the faithfulness of $\calH(F)$ is guaranteed by the faithfulness of $F$. This completes the proof of (1). Statement (2) is straightforward by the definition of the exact structure $\calE^{\rm deg}$.
\end{proof}

Let $\calC$ and $\calD$ be additive categories. Since any additive functor $F:\calC\to \calD$ preserves split exact sequences. It follows immediately from (2) that $\calH(F)$ is an exact functor $\calH(\calC)^{\rm cw}\to \calH(\calD)^{\rm cw}$.

Let $\calC$ be an exact category. Define the {\bf monomorphism category} $\mathcal S(\calC)$ (resp. {\bf epimorphism category} $\mathcal F(\calC)$) as the full subcategory of $(\mathcal H(\calC),\calE^{\rm deg})$, consisting of objects which are inflations (resp. deflations) in $\calC$.  
Notice that 
different exact structures on $\calC$ may result in different monomorphism (epimorphism) categories. 
 
There are a couple of functors:
\begin{eqnarray*}
\bKer:\calF(\calC) \to \calS(\calC),   &\bcok:\calS(\calC) \to\calF(\calC)  \\
\left(\begin{matrix}A\\B\end{matrix}\right)_f  \mapsto \left(\begin{matrix}\ker f\\A\end{matrix}\right)_i,
&\left(\begin{matrix}A\\B\end{matrix}\right)_f  \mapsto  \left(\begin{matrix}B\\\cok f\end{matrix}\right)_p
\end{eqnarray*}
where $i$ is the canonical embedding and $p$ is the canonical projection.
The following result is usually referred to as the Ringel-Schmidmeier-Simson (RSS) equivalence:
\begin{prop}\label{prop:RSS}
 Let $\calC$ be an exact category.
Additive functors $\bKer$ and $\bcok$ form a pair of inverse equivalences.
\end{prop}
As immediate consequences of the RSS equivalence. We see that $\bKer$ and $\bcok$  are exact equivalences which preserve indecomposable objects,  $\Ext$-projectives,  $\Ext$-injectives and almost split sequences.

As above, for an exact functor $F:\calC\to \calD$, we can consider the induced functors $\calS(F):\calS(\calC)\to \calS(\calD)$ and $\calF(F):\calF(\calC)\to \calF(\calD)$, which send $\left(\begin{smallmatrix}A\\B\end{smallmatrix}\right)_f$ to 
$\left(\begin{smallmatrix}FA\\FB\end{smallmatrix}\right)_{Ff}$ and send morphisms $(\alpha,\beta):f\to g$ to $(F\alpha, F\beta):Ff\to Fg$. It is easy to see that $\calS(F)$ and $\calF(F)$ are also exact.

\begin{prop}
    Let $F:\calC\to \calD$ be an exact functor.  If $F$ is an exact equivalence, then $\calS(F)$ and $\calF(F)$ are also exact equivalences. 
\end{prop}
\begin{proof}
    To show that $\calS(F)$ and $\calF(F)$ are fully-faithful is similar to the proof of Proposition \ref{prop:HF_equiv}, which we omit. To see $\calS(F)$ (or $\calF(F))$ is dense, notice that since $F$ is full and dense, for any object  $\left(\begin{smallmatrix}A\\B\end{smallmatrix}\right)_f$ in $\calS(\calD)$ (resp. $\calF(\calD)$), there is an object $\left(\begin{smallmatrix}X\\Y\end{smallmatrix}\right)_u$ in $\calH(\calC)$ such that $\left(\begin{smallmatrix}FX\\FY\end{smallmatrix}\right)_{Fu}\cong\left(\begin{smallmatrix}A\\B\end{smallmatrix}\right)_f$. But $u$ must be a inflation (resp. deflation) as $F$ reflects exact sequences. Clearly, both $\calS(F)$ and $\calF(F)$ preserve and reflect exact sequences. 
\end{proof}

\subsection{Almost split sequences in morphism categories}

In the following, we will characterize the almost split sequences in morphism categories.  

There are some special almost split sequences in the morphism categories which are easy to check by definitions. We call them the {\bf sporadic almost split sequences}.

\begin{lem} \label{lem:sporadic} 
	  Let $\calC$ be an exact category. For objects $A,C\in\calC$, let $p:P(C)\to C$ be the $\Ext$-projective cover and $e:A\to I(A)$ be the $\Ext$-injective envelope. 
Assume $\delta$ is  an almost split sequence in $\calC$, $p'$ is a lift of the  $\Ext$-projective cover $p$ and $e'$ is a lift of the  $\Ext$-injective envelope $e$.
	$$\xymatrix{&&&P(C)\ar[d]^p\ar[dl]_{p'}\\
	\delta:\ \ 0\ar[r]&A\ar[r]^f\ar[d]_e&B\ar[dl]^{e'}\ar[r]^g&C\ar[r]&0\\
	&I(A)}
	$$
	
	\begin{enumerate}

          \item \cite[Lemma 4.5]{HE} The following exact sequences are almost split in ${\calH}(\calC)$.
			$$\xymatrix@1{  0\ar[r] & {\left(\begin{smallmatrix} A\\ A\end{smallmatrix}\right)}_{1}
			\ar[rr]^-{\left(\begin{smallmatrix} 1 \\ f\end{smallmatrix}\right)}
			& & {\left(\begin{smallmatrix}A\\ B\end{smallmatrix}\right)}_{f}\ar[rr]^-{\left(\begin{smallmatrix} 0 \\ g\end{smallmatrix}\right)}& &
			{\left(\begin{smallmatrix}0\\ C\end{smallmatrix}\right)}_{0}\ar[r]& 0. } \ \    $$
            $$\xymatrix@1{  0\ar[r] & {\left(\begin{smallmatrix} A\\ 0\end{smallmatrix}\right)}_{0}
			\ar[rr]^-{\left(\begin{smallmatrix} f \\ 0 \end{smallmatrix}\right)}
			& & {\left(\begin{smallmatrix} B\\ C\end{smallmatrix}\right)}_{g}\ar[rr]^-{\left(\begin{smallmatrix} g \\ 1\end{smallmatrix}\right)}& &
			{\left(\begin{smallmatrix}C \\ C\end{smallmatrix}\right)}_{1}\ar[r]& 0. } \ \    $$

		\item  The following exact sequences are almost split in ${\calS}(\calC)$.
	
        $$\xymatrix@1{  0\ar[r] & {\left(\begin{smallmatrix} A\\ A\end{smallmatrix}\right)}_{1}
			\ar[rr]^-{\left(\begin{smallmatrix} 1 \\ f\end{smallmatrix}\right)}
			& & {\left(\begin{smallmatrix}A\\ B\end{smallmatrix}\right)}_{f}\ar[rr]^-{\left(\begin{smallmatrix} 0 \\ g\end{smallmatrix}\right)}& &
			{\left(\begin{smallmatrix}0\\ C\end{smallmatrix}\right)}_{0}\ar[r]& 0. } \ \    $$
			$$\xymatrix@1{  0\ar[r] & {\left(\begin{smallmatrix} A\\ I(A)\end{smallmatrix}\right)}_{e}
			\ar[rr]^-{\left(\begin{smallmatrix} f \\ {\left(\begin{smallmatrix} 1 \\ 0\end{smallmatrix}\right)}\end{smallmatrix}\right)}
			& & {\left(\begin{smallmatrix}B\\ I(A)\oplus C\end{smallmatrix}\right)}_{\left(\begin{smallmatrix}e'\\\end{smallmatrix}\right)}\ar[rr]^-{\left(\begin{smallmatrix} g \\ {\left(\begin{smallmatrix}1&0\end{smallmatrix}\right)}\end{smallmatrix}\right)}& &
			{\left(\begin{smallmatrix}C\\ C\end{smallmatrix}\right)}_{1}\ar[r]& 0. } \ \    $$		
		\item   The following exact sequences are almost split in ${\calF}(\calC)$.
        $$\xymatrix@1{  0\ar[r] & {\left(\begin{smallmatrix} A\\ 0\end{smallmatrix}\right)}_{0}
			\ar[rr]^-{\left(\begin{smallmatrix} f \\ 0 \end{smallmatrix}\right)}
			& & {\left(\begin{smallmatrix} B\\ C\end{smallmatrix}\right)}_{g}\ar[rr]^-{\left(\begin{smallmatrix} g \\ 1\end{smallmatrix}\right)}& &
			{\left(\begin{smallmatrix}C \\ C\end{smallmatrix}\right)}_{1}\ar[r]& 0. } \ \    $$
			$$\xymatrix@1{  0\ar[r] & {\left(\begin{smallmatrix} A\\ A\end{smallmatrix}\right)}_{1}
			\ar[rr]^-{\left(\begin{smallmatrix} [1,0]^T\\ f \end{smallmatrix}\right)}
			& & {\left(\begin{smallmatrix} A\oplus P(C)\\ B\end{smallmatrix}\right)}_{[f,p']}\ar[rr]^-{\left(\begin{smallmatrix} [0,1]\\ g\end{smallmatrix}\right)}& &
			{\left(\begin{smallmatrix}P(C)\\ C\end{smallmatrix}\right)}_{p}\ar[r]& 0. } \ \    $$		
 	\end{enumerate}
\end{lem}
\begin{proof}
As the proof is essentially the same as in the module category, we refer to \cite[Proposition 7.4]{RS2}.  
\end{proof}

\label{rem:nonspo_cw}
  
 Next, we show that non-sporadic almost split sequences in $\calH(\calC)$  (resp. ${\calS}(\calC)$, ${\calF}(\calC)$) are \rm{cw}-exact.

\begin{lem}\label{cor:dg_slplit_ almost_split}
 Let $\calC$ be an exact category and 
 $$\xymatrix@1{\delta: 0 \ar[r] & {\left(\begin{smallmatrix} X_1\\ X_2\end{smallmatrix}\right)}_{f}
		\ar[rr]^-{\left(\begin{smallmatrix} \phi_1 \\ \phi_2 \end{smallmatrix}\right)}
		& & {\left(\begin{smallmatrix}Z_1\\ Z_2\end{smallmatrix}\right)}_{h}\ar[rr]^-{\left(\begin{smallmatrix} \psi_1 \\ \psi_2\end{smallmatrix}\right)}& &
		{\left(\begin{smallmatrix}Y_1 \\ Y_2\end{smallmatrix}\right)}_{g}\ar[r]& 0  } \ \    $$
be an almost split sequence in ${\calH}(\calC)$ (resp. ${\calS}(\calC)$, ${\calF}(\calC)$). If $\delta$ is not isomorphic to one of the almost split sequences in Lemma \ref{lem:sporadic} (1) (resp. (2), (3)), then the sequences $0 \rt X_i\st{\phi_i}\rt Z_i\st{\psi_i} \rt Y_i\rt 0$, $i=1, 2$, split.
\end{lem}

\begin{proof}
    We just prove for the case $\delta\in\calH(\calC)$. Since $\delta$ is not isomorphic to one of the sequences in Lemma \ref{lem:sporadic} $(1)$, $g$ is not a split monomorphism. It is easy to check that  
    $\left(\begin{smallmatrix} 1 \\ g \end{smallmatrix}\right): \left(\begin{smallmatrix} Y_1\\ Y_1\end{smallmatrix}\right)_{1} \to {\left(\begin{smallmatrix}Y_1\\ Y_2\end{smallmatrix}\right)}_{g}
$ is not a split epimorphism. Hence $\left(\begin{smallmatrix} 1 \\ g \end{smallmatrix}\right)$ factors through ${\left(\begin{smallmatrix}Z_1\\ Z_2\end{smallmatrix}\right)}_{h}$ i.e. the top row splits. Similarly, $f$ is not a split epimorphism. Thus $\left(\begin{smallmatrix} f \\ 1 \end{smallmatrix}\right): \left(\begin{smallmatrix} X_1\\ X_2\end{smallmatrix}\right)_{f} \to {\left(\begin{smallmatrix}X_2\\ X_2\end{smallmatrix}\right)}_{1}
$ is not a split monomorphism and hence factors through ${\left(\begin{smallmatrix}Z_1\\ Z_2\end{smallmatrix}\right)}_{h}$. So the bottom row splits.

A similar argument as \cite[Corollary 7.6]{RS2} shows that if $\delta$ is an almost split sequences in $\calS(\calC)$ or $\calF(\calC)$, which is not in sporadic cases (2) and (3), then both rows of $\delta$ split.
\end{proof}
 
Now we discuss the key technique, which allows us to transfer cw-exact almost split sequences in $\calH(\calC)$ to almost split sequences in $\calC\modd$.

\begin{constr}\label{Construction-Morphism-Functor}\cite[Construction 4.1]{HE}
 Let $(X_1\stackrel{f}\rightarrow X_2)$ be an object of ${\calH}(\calC)$. Define
$$ (X_1\stackrel{f}\rightarrow X_2)\stackrel{\Theta}\mapsto\ \cok\Hom_{\calC}(-,f).$$	
If $\phi=\left(\begin{smallmatrix}
	\phi_1\\ \phi_2
	\end{smallmatrix}\right):X=\left(\begin{smallmatrix}
	X_1\\ X_2
	\end{smallmatrix}\right)_f\rightarrow \left(\begin{smallmatrix}
	Z_1\\ Z_2
	\end{smallmatrix} \right)_{h}=Z$  is a morphism in $\calH(\calC)$,  $\Theta (\phi)$ is defined to be the unique morphism that makes the bottom left square of the following Figure \ref{fig:functor} commutes.
\end{constr}

Let $\mathrm{Y}=(Y_1\st{g}\rt Y_2)$ be an indecomposable object in $\calH(\calC)$ which is neither a split monomorphism nor a split epimorphism. Take
	$$\delta:\,\,\xymatrix@1{  & {\left(\begin{smallmatrix} X_1\\ X_2\end{smallmatrix}\right)}_{f}
		\ar[rr]^-{\left(\begin{smallmatrix} \phi_1 \\ \phi_2 \end{smallmatrix}\right)}
		& & {\left(\begin{smallmatrix}Z_1\\ Z_2\end{smallmatrix}\right)}_{h}\ar[rr]^-{\left(\begin{smallmatrix} \psi_1 \\ \psi_2\end{smallmatrix}\right)}&&
		{\left(\begin{smallmatrix}Y_1 \\ Y_2\end{smallmatrix}\right)}_{g}&  } \ \    $$
to be the almost split sequence in $\calH(\calC)^{{\rm cw}}$ ending at $\mathrm{Y}$.	
For simplicity, set $\mathrm{Z}=(Z_1\st{h}\rt Z_2)$,  $\mathrm{X}=(X_1\st{f}\rt X_2)$, $\phi=(\phi_1, \phi_2)^T$ and $\psi=(\psi_1, \psi_2)^T$.	

Assume $\calC$ has pseudokernel. By
applying $\Theta$, one gets the commutative diagram with exact rows.
	\begin{figure}[h]
 $$
 \xymatrix{%& 0 \ar[d] & 0 \ar[d] & 0 \ar[d]\\
		%0 \ar[r] & K_1 \ar[d] \ar[r] &  K_2 \ar[d]^i
		%\ar[r]^{\eta} & K_3 \ar[d]^{\lambda}&\\
		0 \ar[r] &(-, X_1)\ar[d]^{(-,f)} \ar[r]^{(-,\phi_1)} & (-, Z_1)
		\ar[d]^{(-,h)}\ar[r]^{(-,\psi_1)} & (-, Y_1) \ar[d]^{(-,g)} \ar[r] & 0\\
		0 \ar[r] &(-, X_2) \ar[d] \ar[r]^{(-,\phi_2)} & (-, Z_2)
		\ar[d] \ar[r]^{(-, \psi_2)} & (-, Y_2) \ar[d] \ar[r] & 0\\
		& \Theta(\mathrm{X}) \ar[r]^{\Theta(\phi)}\ar[d] & \Theta(\mathrm{Z}) \ar[r]^{\Theta(\psi)}\ar[d]  & \Theta(\mathrm{Y}) \ar[r]\ar[d] &0\\ & 0  & 0  & 0 &  }
  $$
 	    \caption{}
	    \label{fig:functor}
	\end{figure} 
 in $\calC\modd$ whose bottom row is indeed the image $\Theta(\delta)$ of $\delta$ under the functor $\Theta$.

\begin{thm}\label{almostpreserving}\cite[Lemma 4.8, Theorem 4.9]{HE}\label{thm:EH}
Let $\calC$ be an additive category having peudokernels. In the above diagram, 
\begin{enumerate}
    \item the map $\Theta(\phi)$ is a monomorphism. Thus, $ \Theta(\delta)$ is a short exact sequence in $ \calC\modd.$
\item $\Theta(\delta)$ is an almost split sequence in $\calC\modd $.
\end{enumerate}
\end{thm}

\subsection{Auslander-Reiten translation in functor categories}\label{sec:AR_fun}

Let $\calV$ be a Hom-finite $k$-linear Krull-Schmidt category. Denote by $\calV\Mod$ the category of additive contravariant functors from $\calV$ to abelian groups and $\calV\modd$ the subcategory consisting of finitely presented functors. Then $\calV$ is called a {\bf dualizing $k$-variety} \cite{AR74}
if the $k$-dual functors $D:\calV\Mod\to\calV^{\rm op}\Mod$ and $D:\calV^{\rm op}\Mod\to\calV\Mod$ given by $D(F)(C)=D(FC)$ for every object $C$ of $\calV$ and $F\in\calV\Mod$ or $\calV^{\rm op}\Mod$ induce dualities
\[D:\calV\modd\to\calV^{\rm op}\modd\ \mbox{ and }\ D:\calV^{\rm op}\modd\to\calV\modd.\] 

  Let $\calV$ be a dualizing $k$-variety. It turns out that $\calV\modd$ is an abelian subcategory of $\calV\Mod$ that admits enough projective and enough injective objects \cite[Theorem 2.4]{AR74}.
The following result shows a connection among almost split sequences, Auslander-Reiten duality, and dualizing vareities.
 \begin{thm}\cite[Corollary 4.5]{INP}\label{thm:ass_equiv_ars}
Let $\calC$ be a $k$-linear $\Ext$-finite Krull-Schmidt exact category with enough  $\Ext$-projectives and $\Ext$-injectives. Then the
following conditions are equivalent.
\begin{enumerate}
\item $\calC$ has almost split  sequences.
\item $\calC$ has an Auslander-Reiten-Serre duality.
\item The stable category $\underline \calC$ is a dualizing $k$-variety.
\item The stable category $\overline \calC$ is a dualizing $k$-variety.
\end{enumerate}
\end{thm}

 For a dualizing $k$-variety $\calV$, since the stable category $\calV\textendash\underline{\rm mod}$ is also a dualizing $k$-variety \cite[Proposition 6.2]{AR74}, we can immediately conclude that $\calV\modd$ has almost split sequences.
In the sequel, we will demonstrate how to calculate Auslander-Reiten translations in the functor category $\calV\modd$.

Following \cite{AR74}, we introduce the notion of transpose for finitely presented functors: Let
$$
(-,B)\stackrel{(-,f)}\to(-,A)\to F\to 0
$$
be a minimal projective presentation of $F\in\calV\modd$. Define the {\bf transpose} $\Tr F\in\calV^{\op}\modd$ as $\cok\Hom_\calV(f,-)$ i.e. there is a minimal projective presentation:
$$
(A,-)\stackrel{(f,-)}\to(B,-)\to \Tr F\to 0.
$$

In fact, $\Tr$ yields a functor $\Tr:\calV\pmodd\to\calV^{\op}\pmodd$. We have the following description of the Auslander-Reiten translation in the functor category $\calV\modd$. 

\begin{prop}\cite[Proposition 2.1]{AR94}\label{prop:DTr_fun}
Let $\calV$ be a dualizing $k$-variety. Then $\calV\modd$ has almost split sequences. Furthermore,
if $0\to F\to G\to H\to 0$ is an almost split sequence in $\calV\modd$, then $F\cong D{\rm Tr} H$.
\end{prop}

\section{Existence of almost split sequences}
In this section, we study homological finiteness and the existence of almost split sequences in morphism categories.

\subsection{Homological finiteness}
 Let $\calC$ be a functorially finite exact subcategory of an abelian category $\calA $ with enough $\Ext$-projective and $\Ext$-injective. Our goal is to show that $\calH(\calC)$, $\calS(\calC)$, $\calF(\calC)$ are functorially finite subcategories of $\calH(\calA)$. Furthermore, we will discuss when these categories have almost split sequences.
 First of all, we study the functorial  finiteness of $\calH(\calC)$.

 \begin{lem}\label{lem:homo_contra}
Let $\calC$ be a contravariantly (resp. covariantly) finite subcategory of an abelian category $\calA $. Then $\calH(\calC)$ is a contravariantly (resp. covariantly) finite subcategory of $\calH(\calA )$.
 \end{lem}

 \begin{proof}
 We just show the contravariant finiteness. The covariant finiteness is similar.

 For any morphism $f:A\to B$ in $\calA$. Let $w:Y\to B$ be a right $\calC$-approximation of $B$ and take the pullback. We can obtain the following commutative diagram:
 $$
 \xymatrix{X\ar[r]^u\ar[d]_{v\circ u}&C\ar[r]^g\ar[d]^{\ \ (p.b.)}_v&Y\ar[d]^w\\ A\ar@{=}[r]&A\ar[r]^f&B,}
 $$
where $u$ is a right $\calC$-approximation of $C$. Now we check that $\left(\begin{smallmatrix}vu\\ w\end{smallmatrix}\right): {\left(\begin{smallmatrix}X\\ Y\end{smallmatrix}\right)}_{gu} \to \left(\begin{smallmatrix}A\\ B\end{smallmatrix}\right)_{f}$ is a right $\calH(\calC)$-approximation. Indeed, if there is a morphism $\left(\begin{smallmatrix}\phi_1\\ \phi_2\end{smallmatrix}\right): {\left(\begin{smallmatrix}X'\\ Y'\end{smallmatrix}\right)}_{h} \to \left(\begin{smallmatrix}A\\ B\end{smallmatrix}\right)_{f}$,
 then there is a morphism $s: Y'\to Y$ such that $ws=\phi_2$ following from the fact that $w$ is a right $\calC$-approximation. So $wsh=\phi_2h=f\phi_1$. By the universal property of pullback, we have $t: X'\to C$ such that $sh=gt,\ \phi_1=vt$. Since $u$ is a right $\calC$-approximation of $C$, there is a morphism $r: X'\to X$ such that $t=ur$. Hence $gur=gt=sh$ and $vur=vt=\phi_1$, i.e. $\left(\begin{smallmatrix}\phi_1\\ \phi_2\end{smallmatrix}\right) $ factors through  $\left(\begin{smallmatrix}vu\\ w\end{smallmatrix}\right)$.  
 \end{proof}

Next, we will discuss the homological finiteness of the monomorphism or epimorphism categories. We recall the first result essentially from \cite[Lemma 2.1]{RS2}.

\begin{lem}\label{lem:hom_fin_ab}
    Let $\calA$ be an abelian category, then 
\begin{enumerate}
    \item $\calS(\calA)$ is a covariantly finite subcategory of $\calH(\calA)$.
    \item $\calF(\calA)$ is a contravariantly finite subcategory of $\calH(\calA)$.
\end{enumerate} 
\end{lem}

Now we turn to study the homological finiteness of monomorphism or epimorphism categories of exact subcategories in an abelian category.
 
 \begin{lem}\label{lem:mimo_contra}
 Let $\calC$ be an exact subcategory  of an abelian category $\calA $. \\
 (1) If $\calC$ has enough $\Ext$-injective, then $\calS(\calC)$ is a contravariantly finite subcategory of $\calH(\calC)$.\\
 (2) If $\calC$ has enough $\Ext$-projective, then $\calF(\calC)$ is a covariantly finite subcategory of $\calH(\calC)$.
\end{lem}
 \begin{proof}
 We just prove (1), (2) is similar. Let $f:X\to Y$ be a morphism in $\calC$. Since $\calC$ has enough $\Ext$-injectives, there is an inflation $[f,e]^T: X\to Y\oplus I$ in $\calH(\calC)$, where $I$ is an $\Ext$-injective object in $\calC$. Now we check that
 $$
 \xymatrix{X\ar@{=}[r]\ar[d]^{\left[\begin{smallmatrix}f\\ e\end{smallmatrix}\right]}&X\ar[d]^f\\Y\oplus I\ar[r]^{[1,0]}&Y}
 $$
 is a right $\calS(\calC)$-approximation. Indeed, if there is a morphism $\left(\begin{smallmatrix}\phi_1\\ \phi_2\end{smallmatrix}\right)  : {\left(\begin{smallmatrix}X'\\ Y'\end{smallmatrix}\right)}_{h} \to \left(\begin{smallmatrix}X\\ Y\end{smallmatrix}\right)_{f}$, where $h$ is an inflation, then
there is a morphism $e':Y'\to I$ such that $e'h=e\phi_1$. So $\left(\begin{smallmatrix}\phi_2 h\\ e'h\end{smallmatrix}\right)=\left(\begin{smallmatrix}f\phi_1\\ e\phi_1\end{smallmatrix}\right)$. That is to say, $\left(\begin{smallmatrix}\phi_1\\ \phi_2\end{smallmatrix}\right)$ factors through $\left(\begin{smallmatrix} X\\ Y\oplus I\end{smallmatrix}\right)$.
 \end{proof}

\begin{cor}\label{cor:sc_contr_fin}
(1) If $\calC$ is a contravariantly finite exact subcategory of an abelian category $\calA$ having enough $\Ext$-injective, then $\calS(\calC)$ is a contravariantly finite subcategory of $\calH(\calA)$ and $\calS(\calA)$.\\
(2) If $\calC$ is a covariantly finite exact subcategory of an abelian category $\calA$ having enough $\Ext$-projective, then $\calF(\calC)$ is a covariantly finite subcategory of $\calH(\calA)$ and $\calF(\calA)$.
\end{cor}
\begin{proof}
We prove (1), (2) is similar. Consider the inclusions of subcategories  $\calS(\calC)\subseteq \calH(\calC)\subseteq\calH(\calA)$, %$\calF(\calC)\subseteq \calH(\calC)\subseteq\calH(\m)$. 
By Lemmas \ref{lem:homo_contra} and \ref{lem:mimo_contra}, the contravariant finiteness of $\calS(\calC)$ in $\calH(\calA)$ follows from the fact that contravariant finiteness is transitive. Hence it is clearly contravaraintly finite in $\calS(\calA)$. 
\end{proof}

To show the covariant finiteness of $\calS(\calC)$ and contravariant finiteness of $\calF(\calC)$, we need to apply the RSS-equivalence.

 \begin{lem}\label{lem:sc_cov_fin}
 (1) If $\calC$ is a contravariantly finite exact subcategory of  an abelian category $\calA$ having enough  $\Ext$-injective, then $\calF(\calC)$ is a contravariantly finite subcategory of $\calF(\calA)$.\\
 (2) If $\calC$ is a covariantly finite exact subcategory of an abelian category $\calA$ having enough $\Ext$-projective, then $\calS(\calC)$ is a covariantly finite subcategory of $\calS(\calA )$.
 \end{lem}
\begin{proof}
We proof (2), (1) is similar.

Let $0\to A\stackrel{f}\to B\stackrel{g}\to C\to 0$ be a conflation in $\mathcal C$. We need to construct a left $\calS(\calC)$-approximation of $f$. Due to Corollary \ref{cor:sc_contr_fin}, $\calF(\calC)$ is a covariantly finite subcategory of $\calF(\calA)$. So $g$ admits a left $\calF(\calC)$-approximation $\left(\begin{smallmatrix}\phi_2\\ \phi_3\end{smallmatrix}\right)  : {\left(\begin{smallmatrix}B\\ C\end{smallmatrix}\right)}_{g} \to \left(\begin{smallmatrix}Y\\ Z\end{smallmatrix}\right)_{v}$. Passing to the kernels we obtain a commutative diagram:
$$
\xymatrix{0\ar[r]&A\ar[r]^f\ar[d]^{\phi_1}&B\ar[r]^g\ar[d]^{\phi_2}&C\ar[r]\ar[d]^{\phi_3}&0\\
0\ar[r]&X\ar[r]^u&Y\ar[r]^v&Z\ar[r]&0.}
$$

Now we check that
$\left(\begin{smallmatrix}\phi_1\\ \phi_2\end{smallmatrix}\right)  : {\left(\begin{smallmatrix}A\\ B\end{smallmatrix}\right)}_{f} \to \left(\begin{smallmatrix}X\\ Y\end{smallmatrix}\right)_{u}$ is a left $\calS(\calC)$-approximation. Indeed, let
$\left(\begin{smallmatrix}\alpha\\ \beta\end{smallmatrix}\right)  : {\left(\begin{smallmatrix}A\\ B\end{smallmatrix}\right)}_{f} \to \left(\begin{smallmatrix}M\\ N\end{smallmatrix}\right)_{s}$ be a morphism in $\calS(\calC)$. We have a commutative diagram
$$
\xymatrix{0\ar[r]&A\ar[r]^f\ar[d]^{\alpha}&B\ar[r]^g\ar[d]^{\beta}&C\ar[r]\ar[d]^{\gamma}&0\\
0\ar[r]&M\ar[r]^s&N\ar[r]^{t}&L\ar[r]&0,}$$
with $\left(\begin{smallmatrix}\beta\\ \gamma\end{smallmatrix}\right)$ factors through $\left(\begin{smallmatrix}\phi_2\\ \phi_3\end{smallmatrix}\right)$. This, together with the property of kernel and its universal property, guarantees that $\left(\begin{smallmatrix}\alpha\\ \beta\end{smallmatrix}\right)$ factors through $\left(\begin{smallmatrix}\phi_1\\ \phi_2\end{smallmatrix}\right)$. Hence it is a left $\calS(\calC)$-approximation.
\end{proof}

\begin{thm}\label{thm:sc_ff}
Let $\calC$ be a functorially finite exact subcategory of an abelian category $\calA$ having enough  $\Ext$-projective and $\Ext$-injective. Then $\calS(\calC)$ and $\calF(\calC)$ are functorially finite subcategories of $\calH(\calA)$. 
 
\end{thm}
\begin{proof}
Consider inclusions of subcategories $\calS(\calC)\subseteq \calS(\calA)\subseteq \calH(\calA)$ and $\calF(\calC)\subseteq \calF(\calA)\subseteq \calH(\calA)$. Due to Lemma \ref{lem:hom_fin_ab} and \ref{lem:mimo_contra}, both $\calS(\calA)$ and $\calF(\calA)$ are functorially finite subcategories of $\calH(\calA)$. Then the assertions follows immediately from Corollary \ref{cor:sc_contr_fin} and Lemma \ref{lem:sc_cov_fin} as well as the transitivity of functorially finiteness. 
\end{proof}

A crucial case to be investigated is for $\calV$ being a dualizing $k$-variety. Denote by $\cP(\calV\modd)$ the full subcategory of representable functors in $\calV\modd$. Recall that the Yoneda functor $Y:\calV\to \cP(\calV\modd)$, sending an object $C\in\calV$ to the representable functor $(-, C)$, is an equivalence of additive categories. Therefore a dualizing $k$-variety $\calV$ can be viewed as a subcategory of the abelian category $\calV\modd$.

\begin{cor}
Let $\calV$ be a dualizing $k$-variety. Then categories $\calH(\calV\modd)$, $\calH(\cP(\calV\modd))$, $\calH(\calV)$ are dualzing $k$-varieties as well.
\end{cor}
\begin{proof}
According to \cite[Proposition 4.2]{AR75}, $\calH(\calV\modd)$ is a dualizing $k$-variety. Since $\cP(\calV\modd)$ is functorially finite in $\calV\modd$ (see \cite[Lemma 2.5]{IJ}), it follows by Lemma \ref{lem:homo_contra} that $\calH(\cP(\calV\modd))$ is functorially finite in $\calH(\calV\modd)$. Thus, it is a dualizing $k$-variety. Notice that $\calV\cong \cP(\calV\modd)$ via the Yoneda functor, it follows that $\calH(\calV)\cong\calH(\cP(\calV\modd))$ is a dualizing $k$-variety.
\end{proof}

However, we emphasize that $Y$ is usually not an exact functor. So $\calV$ usually cannot be viewed as an exact subcategory of $\calV\modd$.
In fact, $\cP(\calV\modd)$ is an exact subcategory of $\calV\modd$ with the exact structure containing only the split exact sequences. Let $\calV^{\rm split}$ be the additive category $\calV$ with the exact structure containing only split exact sequences. Then the equivalence $Y:\calV^{\rm split}\to \cP(\calV\modd)$ preserves and reflects exact sequences. Due to Proposition \ref{prop:HF_equiv}, we have the following conclusion.

\begin{prop}\label{prop:yoneda_cw}
 Let $\calV$ be a dualizing variety. Then the Yoneda embedding induces an equivalence
 $ \calH(Y):\calH(\calV)^{\rm cw} \to \calH(\cP(\calV\modd)),$ where $\calH(Y)$ preserves and reflects exact sequences.
\end{prop}

Thus $\calH(\calV)^{\rm cw}$ can be identified with $\calH(\cP(\calV\modd))$ as an exact full subcategory of $\calH(\calV\modd)$ via the functor $\calH(Y)$.

Before we continue to discuss the existence of almost split sequences, it is worth mentioning that the homological finiteness of subcategories and being a dualizing $k$-variety do not depend on the exact structure on these categories i.e. When we say that $\calH(\calC)$ is functorially finite in $\calH(\calA)$, it means that both $\calH(\calC)^{\rm cw}$ and $\calH(\calC)^{\rm deg}$ are functorially finite in $\calH(\calA)$ as an additive subcategory. However, the almost split structures will depend on the exact structures on each category.

\subsection{Existence of almost split sequences}

\begin{thm} \label{C-iff-H(C)}
    Let $\calV$ be a dualizing $k$-variety. Then $\calH(\calV)$ has almost split sequences if and only if $\calV$ has almost split sequences.
\end{thm}
\begin{proof}
  Assume that $\calH(\calV)$ has almost split sequences. For a non-projective object $C\in\calV$,  consider the almost split sequence ending at the object $\left(\begin{smallmatrix}0\\ C\end{smallmatrix}\right)_{0}$.
			$$\xymatrix@1{  0\ar[r] & {\left(\begin{smallmatrix} D\\ A\end{smallmatrix}\right)}_{g}
			\ar[rr]^-{\left(\begin{smallmatrix} 1 \\ \alpha\end{smallmatrix}\right)}
			& & {\left(\begin{smallmatrix}D\\ B\end{smallmatrix}\right)}_{f}\ar[rr]^-{\left(\begin{smallmatrix} 0 \\ \beta\end{smallmatrix}\right)}& &
			{\left(\begin{smallmatrix}0\\ C\end{smallmatrix}\right)}_{0}\ar[r]& 0. } \ \    $$		
It is straightforward to check that the bottom row $0 \rt  A\st{\alpha} \rt B\st{\beta} \rt C \rt 0$ is an almost split sequence. 
Similarly, consider the almost split sequence starting at $\left(\begin{smallmatrix} A \\ 0 \end{smallmatrix}\right)$:
$$\xymatrix@1{  0\ar[r] & {\left(\begin{smallmatrix} A\\ 0\end{smallmatrix}\right)}_{0}
			\ar[rr]^-{\left(\begin{smallmatrix} \alpha \\ 0 \end{smallmatrix}\right)}
			& & {\left(\begin{smallmatrix} B\\ D\end{smallmatrix}\right)}_{f}\ar[rr]^-{\left(\begin{smallmatrix} \beta \\ 1\end{smallmatrix}\right)}& &
			{\left(\begin{smallmatrix}C \\ D\end{smallmatrix}\right)}_{g}\ar[r]& 0. } \ \    $$
It is straightforward to check that the upper row $0 \rt  A\st{\alpha} \rt B\st{\beta} \rt C \rt 0$ is an almost split sequence.             

The converse implication is essentially inferred in the proof of \cite[Proposition 4.6]{HE}: Assuming $\calV$ is a dualizing $k$-variety, then $\calH(\calV)^{\rm cw}$ has almost split sequences due to  \cite[Proposition 4.6]{HE}. Consider an indecomposable non-projective object $\left(\begin{smallmatrix}X \\ Y\end{smallmatrix}\right)_h$ in $\calH(\calV)$. If $h$ is not a split monomorphism, then it is non-projective in $\calH(\calV)^{\rm cw}$. Thus there is an almost split sequence ending at $\left(\begin{smallmatrix}X \\ Y\end{smallmatrix}\right)_h$ in $\calH(\calV)^{\rm cw}$. By Lemma \ref{exact-AR}, it is an almost split sequence in $\calH(\calV)$ too. Otherwise, if $h$ is a split monomorphism, then the almost split sequence ending at $\left(\begin{smallmatrix}X \\ Y\end{smallmatrix}\right)_h$ is given by Lemma \ref{lem:sporadic}. So $\calH(\calV)$ has right almost split sequences. Similarly, one can show that it has left almost split sequences.   
\end{proof}

\begin{cor}\label{lem:Cmod_AR}
    Let $\calV$ be a dualizing $k$-variety. Then $\calH(\calV\modd)$ has almost split sequences.
\end{cor} 
\begin{proof}
    Notice that by Proposition \ref{prop:DTr_fun}, $\calV\modd$ has almost split sequences. Then the assertion follows 
by Theorem \ref{C-iff-H(C)}.
\end{proof}

 We recall a criterion for the existence of almost split sequences due to Liu, Ng and Paquette, which is a generalization of \cite[Theorem 2.4]{AS}.
\begin{lem}\cite[Corollary 3.6]{LNP} \label{C:AR-sequence}
Let $\calA$ be an exact $R$-category with almost split sequences where $R$ is a commutative ring, and $\calC$ be a functorially finite subcategory of $\calA$ and closed under extensions, then $\calC$ has almost split sequences.
\end{lem} 

 Now we present the main result for the existence of almost split sequences in monomorphism categories.

 \begin{thm}\label{thm:sc_ar}
Let $\calV$ be a dualizing $k$-variety and $\calC$ be a functorially finite exact subcategory of $\calV\modd$ having enough $\Ext$-projective and $\Ext$-injective objects. Then $\calS(\calC)$ and $\calF(\calC)$ have almost split sequences.  
\end{thm}
\begin{proof}
    Due to Corollary \ref{lem:Cmod_AR}, $\calH(\calV\modd)$ has almost split sequences. Also due to Theorem \ref{thm:sc_ff}, $\calS(\calC)$ and $\calF(\calC)$ are functorially finite, extension closed subcategories of $\calH(\calV\modd)$. Thus by Lemma \ref{C:AR-sequence}, they have almost split sequences.
\end{proof}
 
\section{Minimal monomorphisms and the stable category}\label{sec:mimo}

Throughout this section, let $\calC$ be a Krull-Schmidt exact category with enough  $\Ext$-injectives. We will introduce the notion of minimal monomorphisms for morphisms in $\calC$ and $\overline\calC$. 
One should compare these results with \cite[\S4\&\S5]{RS2}, where the minimal monomorphisms are defined in the module categories. We remind the reader about the main difference here: in an exact category, a morphism does not necessarily have a kernel.  
 
\subsection{The minimal monomorphism functor}
 Let $f:A\to B\in \calH(\calC)$. As shown in Lemma \ref{lem:mimo_contra}, $f$ has a right $\calS(\calC)$-approximation $\Mimo f:A\stackrel{[f,e]^T}\to B\oplus I_A$, where $e:A\to I_A$ is an arbitrary monomorphism with $I_A$ $\Ext$-injective. In addition, there is an exact sequence $$0\to i_A\to \Mimo f \to f\to 0,$$ where $i_A: 0\to I_A$ is an $\Ext$-injective object in $\calS(\calC)$. Notice that the construction of $\Mimo f$ as an object in $\calS(\calC)$ depends on the choice of $e$.
 
 \medskip

Let $\alpha: {\left(\begin{smallmatrix} A\\ B\end{smallmatrix}\right)}_{f}\to {\left(\begin{smallmatrix} C\\ D\end{smallmatrix}\right)}_{g}$ be a morphism in $\calH(\calC)$. For fixed $\Mimo f $ and $\Mimo g $, since $\Mimo g $ is a right $\calS(\calC)$-approximation of $g$, there is a homomorphism $\Mimo(\alpha):\ \Mimo f \to \Mimo g $ such that the following diagram commutes, 
 \begin{align*}
% \tag{i}
\xymatrix { 0 \ar[r] & i_{A}\ar[d]^{h} \ar[r]^{\sigma_f\qquad }&\Mimo f \ar[r]^{\qquad \pi_f}\ar[d]^{ \Mimo(\alpha)}&f \ar[r] \ar[d]^{\alpha}& 0    \\ 0 \ar[r] & i_{C} \ar[r]^{\sigma_g \qquad}&\Mimo g \ar[r]^{\qquad\pi_g} &g \ar[r] & 0.   }
\end{align*}
 Also notice that $\Mimo(\alpha)$ is not uniquely determined in $\calS(\calC)$. So $\Mimo$ is not a functor $\calH(\calC)\to \calS(\calC)$. However, it does give rise to a functor $\calH(\calC)\to \overline{\calS(\calC)}$, which we are going to define now.

\begin{propdef}\label{mono-functor}
 For any object $f\in\calH(\calC)$ and any morphism $\alpha:f\to g\in\calH(\calC)$, the construction $\Mimo$ gives rise to a unique object $\Mimo f $ and a unique morphism $\Mimo(\alpha)$ up to isomorphisms in $\overline{\calS(\calC)}$. Hence $\Mimo: \calH(\calC)\to \overline{\calS(\calC)}$ is a functor.
\end{propdef}

\begin{proof} (1) We show that $\Mimo f $ is uniquely determined in $\overline{\calS(\calC)}$. In fact, it is easy to see that for two monomorphisms $e:A\to I_A$ and $e':A\to I'_A$, $${\left(\begin{smallmatrix} A\\ B\oplus I_A\end{smallmatrix}\right)}_{\left[\begin{smallmatrix} f\\ e\end{smallmatrix}\right]}\oplus\left(\begin{smallmatrix} 0\\ I'_A\end{smallmatrix}\right)\cong 
{\left(\begin{smallmatrix} A\\ B\oplus I_A\oplus I'_A\end{smallmatrix}\right)}_{\left[\begin{smallmatrix} f\\ e\\e'\end{smallmatrix}\right]}\cong {\left(\begin{smallmatrix} A\\ B\oplus I'_A\end{smallmatrix}\right)}_{\left[\begin{smallmatrix} f\\ e'\end{smallmatrix}\right]}\oplus\left(\begin{smallmatrix} 0\\ I_A\end{smallmatrix}\right).$$
Hence ${\left(\begin{smallmatrix} A\\ B\oplus I_A\end{smallmatrix}\right)}_{\left[\begin{smallmatrix} f\\ e\end{smallmatrix}\right]}\cong {\left(\begin{smallmatrix} A\\ B\oplus I'_A\end{smallmatrix}\right)}_{\left[\begin{smallmatrix} f\\ e'\end{smallmatrix}\right]}$ in the stable category $\overline {\calS(\calC)}$, i.e. $\Mimo f $ is unique in  $\overline {\calS(\calC)}$.

(2) We show that for fixed $\Mimo f $ and $\Mimo g $, the morphism $\Mimo(\alpha)$ is unique. In fact, if there is another morphism $\Mimo(\alpha)'$ making the above diagram commutative, then $\pi_g(\Mimo(\alpha)-\Mimo(\alpha)')=0$. So  $\Mimo(\alpha)-\Mimo(\alpha)'$ factors through the injective object $i_C$, i.e. $\Mimo(\alpha)=\Mimo(\alpha)'$ in $\overline{\calS(\calC)}$.

(3) We show that $\Mimo$ respects the composition of morphisms. Let $\alpha:\ f\to g$ and $\beta:\ g\to h$ be homomorphisms in $\calH(\calC)$. In the above construction, we have $\Mimo(\beta\alpha)-\Mimo(\beta)\Mimo(\alpha)$ factors through an injective object. Hence $\Mimo(\beta\alpha)=\Mimo(\beta)\Mimo(\alpha)$ in $\overline{\calS(\calC)}$.
\end{proof}

Let $\calS(\calC)_\calI$ be the full subcategory of $\calS(\calC)$ consisting of objects which have no non-zero  $\Ext$-injective direct summands $I=I$ or $0\to I$.
According to the dual of Corollary \ref{cor:qc_epiv}, there is a bijection between indecomposable objects in $\calS(\calC)_\calI$ and $\overline{\calS(\calC)}$. For convenience, we introduce a similar concept.

\begin{defn}
 For each $f\in\calH(\calC)$, we denote by $\mimo f$ the unique object in $\calS(\calC)_\calI$ such that under the costable functor it is isomorphic to $\Mimo f$.
\end{defn} 

Notice that $\mimo$ is not a functor. We summarize some facts about minimal monomorphisms, which will be useful for the computation.
 \begin{prop}\label{prop:mimo} \ 
\begin{enumerate}
\item $\Mimo(f\oplus g)\cong\Mimo f\oplus \Mimo g$.
\item For a morphism $f:A\to B$ and any morphism $h:A\to I$ with $I$ an $\Ext$-injective object, if $g=[f,h]^T:A\to B\oplus I$ is a monomorphism, then $\Mimo f\cong g$ in $\overline{\calS(\calC)}$.
  \item  If $f$ has a kernel and $i:\ker f\to I(\ker f)$ is the $\Ext$-injective envelope, then $\mimo f\cong [f,e]^T:A\to B\oplus I(\ker f)$, where $e$ is an extension of the $\Ext$-injective envelope $i$.
  \item   If $f$ is a monomorphism without non-zero $\Ext$-injective direct summands, then $\mimo f= f$. 
\end{enumerate}
 \end{prop}

We conclude that the functor $\Mimo$ is dense. In fact, one can go a little bit further. Recall that $\calH(\calC_\calI)=\{\left(\begin{smallmatrix}A\\B\end{smallmatrix}\right)_f\in\calH(\calC): A,B$  has no non-zero injective summands$\}$.
\begin{lem}\label{lem:hi_dense}
 The functor $\Mimo$ restricted on $\calH(\calC_\calI)$ is dense.
\end{lem}
\begin{proof} 
Let $\left(\begin{smallmatrix}A\\B\end{smallmatrix}\right)_f\in \overline{\calS(\calC)}$. Up to isomorphism, we can assume that $A$ has no non-zero  $\Ext$-injectives direct summands. Otherwise, assume $f=[f_1,f_2]:A_1\oplus I_1\to B$, where $I_1$ is  $\Ext$-injective. Since $f$ is a monomorphism, it follows that $\left(\begin{smallmatrix}I_1\\I_1\end{smallmatrix}\right)_1$ is a direct summand of $f$.  Hence $f\cong f_1$ in $\overline{\calS(\calC)}$. 

Now assume $f=[f',f'']^T:A\to B_1\oplus J$, where $A$ and $B_1$ have no non-zero $\Ext$-injective  direct summands and $J$ is $\Ext$-injective. Then $\Mimo f' \cong f$ by Proposition \ref{prop:mimo} (2).
\end{proof}

\subsection{The epivalence} 
 The aim of this section is to define minimal monomorphisms of objects in $\calH(\overline\calC)$ via an epivalence $\overline{\calS(\calC)}\to \calH(\overline\calC)$. We start from discussing the relations among various morphism categories. 

 The costable functor $Q^\calC:\calC\to \overline\calC$ induces $\calH(Q^\calC):\calH(\calC)\to \calH(\overline\calC)$ (see Section \ref{sec:fun_morph}).
  It restricts to  a functor $\calS(\calC)\to  \calH(\overline\calC)$ and then induces a functor $\pi: \overline{\calS(\calC)}\to  \calH(\overline\calC)$. It is easy to check that there is a diagram of functors with commutative squares and triangles:
  
 \begin{figure}[h]
     \centering
  \begin{tikzcd}
 % \calS(\calC)_\calI\ar[rr,hook] &&
  \calS(\calC)\ar[rr,"Q^{\calS(\calC)}"]\ar[d,hook] &&\overline{\calS(\calC)}\ar[d,"\pi"]\\
  %\calH(\calC_\calI)\ar[rr,hook] &&
  \calH(\calC)\ar[rr,"\calH(Q^\calC)"'] \ar[urr,"\Mimo"]%\ar[ull,dotted,"\mimo"']
  &&{\calH(\overline\calC).}
 \end{tikzcd}    
     \caption{Functors between morphism categories}
     \label{fig:costable}
 \end{figure}

\begin{lem}\label{lem:mimo_eq} {\normalfont [An equivalent condition for $\Mimo f\cong\Mimo g$]}
 For morphisms $\left(\begin{smallmatrix}A\\B\end{smallmatrix}\right)_f$ and $\left(\begin{smallmatrix}C\\D\end{smallmatrix}\right)_g$ in $ \calC_\calI$, the following statements are equivalent:
\begin{enumerate}
 \item $\overline f\cong \overline g$ in $\calH(\overline\calC)$, 
\item there are isomorphisms $a:A\to C$ and $b:B\to D$ in $\calC$, such that $\beta f-g\alpha$ factors through an injective object,
\item $\Mimo f\cong \Mimo g$. 
\end{enumerate} 
\end{lem} 
\begin{proof}
    $(1)\implies (2)$: Assume $\overline f\cong \overline g$. There are morphisms $a:A\to C$ and $b:B\to D$ such that (i) $\overline a$ and $\overline b$ are isomorphisms  in $\overline\calC$ and (ii) $\overline{bf}=\overline{ga}$. By the dual of Lemma \ref{lem:cp_iso},  (i) implies that $a$, $b$ are isomorphisms. And (ii) implies that $bf-ga$ factors through an $\Ext$-injective object. \\
    $(2)\implies (3)$: Assume $ga-bf:A\stackrel{t}\to Q\stackrel{s}\to D$, where $Q$ is $\Ext$-injective. Given a monomorphism $[f,e]^T:A\to B\oplus I$, it follows that $t$ can factor through $[f,e]^T$, i.e. there is $[u,v]:B\oplus I\to Q$ such that $t=uf+ve$. Since $B$ has no non-zero $\Ext$-injective summands, the composition $b^{-1}su\in\rad\End(B)$. So $b+su=b(1+b^{-1}su)$ is an isomorphism. Hence, the commutative diagram below shows an isomorphism $[f,e]^T\cong [g,ea^{-1}]^T$
$$
\xymatrix{A\ar[r]^{[f,e]^T}\ar[d]_{a}&B\oplus I\ar[d]^{\left[\begin{smallmatrix}b+su&sv\\0&1\end{smallmatrix}\right]}\\ C\ar[r]^{[g,ea^{-1}]^T}&D\oplus I}
$$
Therefore by Proposition \ref{prop:mimo} $\Mimo f\cong [f,e]^T\cong [g,ea^{-1}]^T\cong \Mimo g$.\\
$(3)\implies (1)$ is trivial.
\end{proof}

 \begin{thm}\label{thm:epivalence}
The functor $\pi: \overline{\calS(\calC)}\to \calH(\overline \calC)$ is an epivalence.
 \end{thm}
 
\begin{proof}
Notice that $\calH(Q^\calC)=\pi\circ \Mimo$. Then $\pi$ is dense following from the fact that $\calH(Q^\calC)$ is dense. Also $\pi$ is full following from the fact that $\calH(Q^\calC)$ is full and $\Mimo$ is dense.

It remains to show that $\pi$ reflects isomorphisms. Indeed, let $f,g\in\overline{\calS(\calC)}$. Assume $\pi(f)\cong \pi(g)$. 
By Lemma \ref{lem:hi_dense}, we can assume $f\cong \Mimo f'$ and $g\cong \Mimo g'$ for some $f',g'\in\calH(\calC_\calI)$.
Thus $\pi (f)\cong \pi (g)$ implies $\overline {f'}\cong \overline {g'}$ in $\calH(\overline \calC)$. Then due to Lemma \ref{lem:mimo_eq}, it follows that $\Mimo f'\cong \Mimo g'$. Therefore, $f\cong g$ in $\overline{\calS(\calC)}$.
\end{proof}

\begin{cor}\label{cor:epivalence}
   The functor $\pi$ induces a bijection between indecomposable objects in $\overline{\calS(\calC)}$ and $\calH(\overline \calC)$.  
\end{cor}

\begin{cor}\label{cor:epivalence2}
    The composition $\pi\circ Q^{\calS(\calC)}:\calS(\calC)\to \calH(\overline\calC)$ is full and dense.
\end{cor}

\begin{rem}\label{rem:mimo_well_def}
 Now we can define the minimal monomorphisms of objects in $\calH(\overline{\calC})$, which will be crucial for justifying the well-definedness of the Auslander-Reiten translation formula in next section. According to Theorem \ref{thm:epivalence}, for an object $\overline f\in\calH(\overline\calC)$, up to isomorphism there exists a unique object $u\in\overline{\calS(\calC)}$, such that $\pi(u)\cong \overline f$. In fact, taking $f$ a representative of $\overline f$ in $\calH(\overline\calC)$, then  $u\cong\Mimo f$. We define $\Mimo \overline f:=\Mimo f$ and $\mimo \overline f=\mimo f$. We also warn the reader that neither $\Mimo$ nor $\mimo$ is a functor on $\calH(\overline\calC)$.
\end{rem}

\begin{lem}\label{lem:mimo_stable}
   For any object $f\in\overline{\calS(\calC)}$, $f\cong\Mimo \overline f$. For any object $f\in \calS(\calC)_\calI$, $f\cong\mimo\overline f$.
\end{lem}
 
\subsection{Minimal epimorphisms}
At the end of this section, we introduce the dual notion of minimal monomorphisms. Let $\calD$ be a Krull-Schmidt exact category with enough $\Ext$-projectives. Then the minimal epimorphism functor $\Mepi:\calH(\calD)\to \underline{\calF(\calD)}$ is defined on objects as $$\Mepi f=[f,p]: A\oplus P_B\to B,$$ where $p:P_B\to B$ is an epimorphism with $P_B$  $\Ext$-projective. 
Similarly, the functor $\pi':\underline{\calF(\calD)}\to \calH(\underline\calD)$ sending $f$ to $\underline f$ is an epivalence. Hence we can define $\Mepi \underline f =\Mepi f$ and $\mepi \underline f=\mepi f$, where $f$ is a representative of $\underline f$ in $\calH(\underline{\calC})$. While we do not state all the dual statements about the properties of minimal epimorphisms, we will feel free to use them when necessary.

 \section {The Auslander-Reiten translation formula}
In this section, we are going to prove the Auslander-Reiten translation formula for monomorphism and epimorphism categories. We begin with a detailed investigation of cw-exact sequences in the morphism categories.

\subsection{CW-exact sequences}
Let $\calA$ be an additive category. For morphisms $a:A_1\to A_2$, $b: B_1\to B_2$ and $c:B_1\to A_2$ in $\calA$, denote by $\delta^{\rm cw}_{a,c,b}$ the {\rm cw}-exact sequence in $\calH(\calA)^{\rm cw}$:
$$
\xymatrix{0\ar[r]&A_1\ar[r]^{[1,0]^T}\ar[d]^a& A_1\oplus B_1\ar[r]^{[0,1]}\ar[d]^{\left[\begin{smallmatrix}a&c\\0&b\end{smallmatrix}\right]} &B_1\ar[r]\ar[d]^b&0\\
0\ar[r]&A_2\ar[r]^{[1,0]^T}& A_2\oplus B_2\ar[r]^{[0,1]} &B_2\ar[r]&0.}
$$

\begin{lem}\label{lem:cw_eq}
Two exact sequences $\delta^{\rm cw}_{a,c,b}$ and $\delta^{\rm cw}_{a,c',b}$ are isomorphic if and only if $c-\alpha c'\beta\in\Htp(a,b)$ for some automorphisms $\alpha\in\Aut(A_2)$, $\beta\in\Aut(B_1)$.
\end{lem}

\begin{proof}

Exact sequences $\delta^{\rm cw}_{a,c,b}$ and $\delta^{\rm cw}_{a,c',b}$ are isomorphic if and only if
 there are isomorphisms:
$$
\left(\begin{smallmatrix}\alpha_1\\\alpha_2\end{smallmatrix}\right):\left(\begin{smallmatrix}A_1\\A_2\end{smallmatrix}\right)_a\to \left(\begin{smallmatrix}A_1\\A_2\end{smallmatrix}\right)_a,
\left(\begin{smallmatrix}\phi_1\\\phi_2\end{smallmatrix}\right):\left(\begin{smallmatrix}A_1\oplus B_1\\A_2\oplus B_2\end{smallmatrix}\right)_{\left[\begin{smallmatrix}a&c\\0&b\end{smallmatrix}\right]}\to \left(\begin{smallmatrix}A_1\oplus B_1\\A_2\oplus B_2\end{smallmatrix}\right)_{\left[\begin{smallmatrix}a&c'\\0&b\end{smallmatrix}\right]},
\left(\begin{smallmatrix}\beta_1\\\beta_2\end{smallmatrix}\right):\left(\begin{smallmatrix}B_1\\B_2\end{smallmatrix}\right)_b\to \left(\begin{smallmatrix}B_1\\B_2\end{smallmatrix}\right)_b,
$$
satisfying $\left(\begin{smallmatrix}[1,0]^T\\ [1,0]^T\end{smallmatrix}\right)\left(\begin{smallmatrix}\alpha_1\\\alpha_2\end{smallmatrix}\right)=\left(\begin{smallmatrix}\phi_1\\\phi_2\end{smallmatrix}\right)\left(\begin{smallmatrix}[1,0]^T\\ [1,0]^T\end{smallmatrix}\right)$ and $\left(\begin{smallmatrix}\beta_1\\\beta_2\end{smallmatrix}\right)\left(\begin{smallmatrix}[0,1]\\ [0,1]\end{smallmatrix}\right)=\left(\begin{smallmatrix}[0,1]\\ [0,1]\end{smallmatrix}\right)\left(\begin{smallmatrix}\phi_1\\\phi_2\end{smallmatrix}\right)$.  It follows that
$$\phi_1=\left(\begin{smallmatrix} \alpha_1& s\\ 0&\beta_1\end{smallmatrix}\right), \phi_2=\left(\begin{smallmatrix} \alpha_2& -t\\0&\beta_2\end{smallmatrix}\right)$$ for some morphisms $s:B_1\to A_1$ and $t:B_2\to A_2$.

The fact that $\left(\begin{smallmatrix}\alpha_1\\\alpha_2\end{smallmatrix}\right), \left(\begin{smallmatrix}\phi_1\\\phi_2\end{smallmatrix}\right), \left(\begin{smallmatrix}\beta_1\\\beta_2\end{smallmatrix}\right)$ are morphisms implies that $a\alpha_1=\alpha_2a$, $b\beta_1=\beta_2b$ and
$$
\left(\begin{smallmatrix} \alpha_2& -t\\ 0&\beta_2\end{smallmatrix}\right)\left(\begin{smallmatrix} a& c\\ 0&b\end{smallmatrix}\right)=
\left(\begin{smallmatrix} a& c'\\ 0&b\end{smallmatrix}\right)\left(\begin{smallmatrix} \alpha_1& s\\ 0&\beta_1\end{smallmatrix}\right).
$$
That is $\alpha_2 c-c'\beta_1=as+tb$. So $c-(\alpha_2)^{-1}c'\beta_1=\alpha_2^{-1}as+\alpha_2^{-1}tb=a\alpha_1^{-1}s+\alpha_2^{-1}tb\in\Htp(a,b)$.
\end{proof}
 
Given two morphisms $a:A_1\to A_2$ and $b:B_1\to B_2$ in $\calA$, denote by $\Htp(a,b):=\Ima\Hom(B_1,a)+\Ima\Hom(b,A_2)$ a subgroup of $\Hom(B_1,A_2)$. Explicitly, $\Htp(a,b)=\{as+tb:$ for some $s:B_1\to A_1$ and $t:B_2\to A_2\}$.
We point out some special cases of Lemma \ref{lem:cw_eq}, which will be useful later.

\begin{cor}\label{cor:cw_split}\ \\
(1) If $c-c'\in\Htp(a,b)$, then $\delta^{\rm cw}_{a,c,b}$ is isomorphic to $\delta^{\rm cw}_{a,c',b}$.\\
(2) The exact sequence $\delta^{\rm cw}_{a,c,b}$ is split if and only if $c\in\Htp(a,b)$.
\end{cor}

\begin{lem}\label{lem:stable almost split}
If there is an almost split sequence in $\calF(\calC)$ which is not a sporadic almost split sequence.
 \begin{align}\label{eq:epi}
  \tag{$*$}
\xymatrix{0\ar[r] &A_1\ar[r]^{\alpha_1}\ar[d]^{f} &B_1\ar[r]^{\beta_1}\ar[d]^{g} &C_1\ar[r]\ar[d]^{h} &0\\
0\ar[r] &A_2\ar[r]^{\alpha_2}&B_2\ar[r]^{\beta_2} &C_2\ar[r]&0,\\
 }
 \end{align}
 then there is an almost split sequence in $\calH(\underline\calC)^{\rm cw}$:
 \begin{align}\label{eq:stable}
   \tag{$**$}
\xymatrix{0\ar[r] &A_1\ar[r]^{\underline{\alpha_1}}\ar[d]^{\underline f} &B_1\ar[r]^{\underline{\beta_1}}\ar[d]^{\underline g} &C_1\ar[r]\ar[d]^{\underline h} &0\\
0\ar[r] &A_2\ar[r]^{\underline{\alpha_2}} &B_2\ar[r]^{\underline{\beta_2}} &C_2\ar[r] &0.
 }
 \end{align}
Moreover, the end term $\underline h$ in (\ref{eq:stable}) is not of type (a) $0\to M$, (b) $ M= M$ or (c) $ M\to 0$.
\end{lem}

\begin{proof}
It suffices to show that $\underline f$ and $\underline h$ are indecomposable and ${\left(\begin{smallmatrix}\underline{\beta_1} \\ \underline{\beta_2}\end{smallmatrix}\right)} $ is right almost split.

Notice that both $f$ and $h$ are  indecomposable non-$\Ext$-projective objects in $\calF(\calC)$. According to Corollary \ref{cor:qc_epiv}, they are indecomposable in $\underline{\calF(\calC)}$. Hence by the dual of Theorem \ref{thm:epivalence}, both $\underline f$ and $\underline h$ are indecomposable in $\calH(\underline\calC)$.

It is easy to see the sequence is {\rm cw}-exact. We claim that it is also non-split. In fact, without loss of generality, we can assume the exact sequence (\ref{eq:epi}) equals $\delta^{\rm cw}_{f,c,h}$, i.e. $g=\left(\begin{smallmatrix}  f&  c\\0& h\end{smallmatrix}\right)$ (see Remark \ref{rem:cw_tri_form}). If the exact sequence (\ref{eq:stable}) were split, then by Corollary \ref{cor:cw_split}, $\underline c\in \Htp(\underline f, \underline h)$. That is, there exists a morphism $c':B_1\to A_2$ which factors through  $\Ext$-projectives such that $c-c'\in\Htp(f,h)$. But since $f$ is an epimorphism, $c'$ can factor through $f$ and hence $c'\in\Htp(f,h)$. Thus, $c\in\Htp(f,h)$ and the exact sequence (\ref{eq:epi}) splits, which is a contradiction.

We claim that ${\left(\begin{smallmatrix}\underline{\beta_1} \\ \underline{\beta_2}\end{smallmatrix}\right)} $ is right almost split. Let ${\left(\begin{smallmatrix}\underline{u} \\ \underline{v}\end{smallmatrix}\right)}:{\left(\begin{smallmatrix}X \\ Y\end{smallmatrix}\right)}_{\underline{\gamma}}\to {\left(\begin{smallmatrix}C_1 \\ C_2\end{smallmatrix}\right)}_{\underline{h}}$ be a non-retraction in $\calH(\underline\calC)$. By the dual of Corollary \ref{cor:epivalence2}, there exists ${\left(\begin{smallmatrix}{u} \\ {v}\end{smallmatrix}\right)}:{\left(\begin{smallmatrix}X \\ Y\end{smallmatrix}\right)}_{ \gamma}\to {\left(\begin{smallmatrix}C_1 \\ C_2\end{smallmatrix}\right)}_{h}$ such that under the stable functor $\pi'Q^{\calF(\calC)}:\calF(\calC)\to \calH(\underline\calC)$, $\pi'Q^{\calF(\calC)}\left(\begin{smallmatrix}{u} \\ {v}\end{smallmatrix}\right)=\left(\begin{smallmatrix}\underline{u} \\ \underline{v}\end{smallmatrix}\right)$. Since the stable functor $\pi'Q^{\calF(\calC)}$ preserves retractions, the map $\left(\begin{smallmatrix}{u} \\ {v}\end{smallmatrix}\right)$ is not a retraction. Hence it factors through $\left(\begin{smallmatrix}B_1 \\ B_2\end{smallmatrix}\right)_g$ and therefore $\left(\begin{smallmatrix}\underline{u} \\ \underline{v}\end{smallmatrix}\right)$ factors through $\left(\begin{smallmatrix}B_1 \\ B_2\end{smallmatrix}\right)_{\underline g}$.
Hence ${\left(\begin{smallmatrix}\underline{\beta_1} \\ \underline{\beta_2}\end{smallmatrix}\right)} $ is right almost split. 

Last we show the ``Moreover'' part.
$(1)$ Assume that $\underline h$ is of the type (a), i.e. $h:C_1\to C_2$ for some $\Ext$-projective object $C_1\in \calC$.
Due to the indecomposability of $h$, it must be  an $\Ext$-projective cover $h:P(C_2)\to C_2$ for some indecomposable non-$\Ext$-projective object $C_2$, which means (\ref{eq:epi}) is of sporadic case (2).
 
$(2)$ Assume that $\underline h$ is  of the type (b). The surjectivity and indecomposability of $h$ implies that $h$ is isomorphic to $C=C$, which means $(*)$ is of sporadic case (1).  

$(3)$ Assume that $\underline h$ is  of the type (c), i.e. $h:C_1\to C_2$ for some $\Ext$-projective object $C_2\in\calC$. It implies that $h$ is isomorphic to an $\Ext$-projective object $P=P$, which is a contradiction.\end{proof}

\subsection{Proof of Theorem \ref{thm:A}}\label{subsec:proof}
First of all, we remark that the formulas are well-defined. In fact, for the first formula, since $\tau_\calC$ is a functor $\underline \calC\to \overline \calC$, $\tau_\calC\underline{\bcok (f)}$ is a morphism in $\overline \calC$. According to Remark \ref{rem:mimo_well_def}, $\Mimo\tau_\calC\underline{\bcok (f)}$ is well-defined. For the second one, $\Mimo\tau_\calC (\underline f)$ is an object in $\overline{\calS(\calC)}$. Notice that the RSS-equivalence functor $\bcok:\calS(\calC)\to \calF(\calC)$ preserves $\Ext$-injective objects. So it can be regarded as a functor $\overline{\calS(\calC)}\to \overline{\calF(\calC)}$, and hence it is well-defined.

Next, realize that it suffices to prove the formula for $\tau_\calF$. The formula for $\tau_\calS$ holds subsequently by RSS-equivalence, which preserves almost split sequences. 

It is straightforward to check that $\tau_\calF f=\cok\Mimo\tau_\calC(\underline f)$ holds for objects $f$ of the forms $C\to C$, $C\to 0$ and $P(C)\to C$ by Lemma \ref{lem:sporadic}. In below, we will show the formula also holds for indecomposable objects $C_1\to C_2$ in $\calF(\calC)$ not of the above three forms. In particular, an almost split sequence in $\calF(\calC)$ ending at $C_1\to C_2$ will be {\rm cw}-exact, due to Remark \ref{rem:nonspo_cw}. 

Note that since $\calC$ has enough  $\Ext$-projectives, the stable category $\underline\calC$ has pseudo-kernel. Hence $\underline\calC\modd$ is an abelian category. Applying Theorem \ref{thm:EH} to the sequence (\ref{eq:stable}) in $\calH(\underline\calC)^{\rm cw}$, we obtain the following commutative diagram in $\underline\calC\modd$ with exact rows and columns such that the bottom row is an almost split sequence. We abbreviate $\Hom_{\underline\calC}(-,\underline f)$ for $(-,\underline f)$.
$$\xymatrix{
		0 \ar[r] &(-, A_1)\ar[d]^{(-,\underline f)} \ar[r]^{(-,\underline{\alpha_1})} & (-, B_1)
		\ar[d]^{(-,\underline g)}\ar[r]^{(-,\underline{\beta_1})} & (-, C_1) \ar[d]^{(-,\underline h)} \ar[r] & 0\\
		0 \ar[r] &(-, A_2) \ar[d] \ar[r]^{(-,\underline{\alpha_2})} & (-, B_2)
		\ar[d] \ar[r]^{(-, \underline{\beta_2})} & (-, C_2) \ar[d] \ar[r] & 0\\
		0\ar[r] &  F \ar[r]\ar[d] & G \ar[r]\ar[d]  & H \ar[r]\ar[d] &0\\ & 0  & 0  & 0 &  }
	$$

  Since $\calC$ is a functorially finite exact subcategory $\calV\modd$ having enough $\Ext$-projective and $\Ext$-injective  objects where $\calV$ is a dualizing $k$-variety, it has almost split sequences by Lemma \ref{C:AR-sequence}. Hence by Theorem \ref{thm:ass_equiv_ars}, $\underline\calC$ is a dualizing variety. Using methods discussed in section \ref{sec:AR_fun}, we can calculate $D\Tr H\cong F$ as below:
 
$$\xymatrix{0\ar[r]&D{\rm Tr}H\ar[r]\ar[d]_\wr& D( C_1,-)\ar[r]^{D(\underline h,-)}\ar[d]_\wr& D(C_2,-)\ar[d]_\wr\\
0\ar[r]&F\ar[r]&\Ext^1_\calC(-,\tau_\calC C_1)\ar@< 2pt>[r]^{\Ext^1_\calC(-,\tau_\calC \underline h)}& \Ext^1_\calC(-,\tau_\calC C_2).}$$

Denote by $p=\mimo \tau_\calC \underline h$. We claim it is an indecomposable object in $\calS(\calC)_\calI$. In fact, since $\underline h$ is indecomposable in $\calH(\underline \calC)$, by Proposition \ref{prop:HF_equiv}, $\tau_\calC\underline h$ is also indecomposable in $\calH(\underline\calC)$. Hence by Corollary \ref{cor:qc_epiv} and \ref{cor:epivalence2}, $\mimo\tau_\calC\underline h$ is indecomposable as well.

Consider the exact sequence:
$$
0\to \tau_\calC C_1\stackrel{p}\to \tau_\calC C_2\oplus I\stackrel{q}\to M\to 0.
$$
According to \cite[Theorem 1.26]{INP}, it induces a long exact sequence:
$$
\Hom_{\underline \calC}(-,\tau_\calC C_2\oplus I) \xrightarrow{(-,\underline q)} \Hom_{\underline \calC}(-, M)\stackrel{\theta}\to \Ext^1_{\calC}(-,\tau_\calC C_1) \xrightarrow{\Ext^1_{\calC}(-,p)}   \Ext^1_{\calC}(-,\tau_\calC C_2).
$$
As $\Ext^1_\calC(-,p)\cong\Ext^1_\calC(-,\tau_\calC \underline h)$, it follows that $F\cong \Ima \theta$. Hence $F$ has a projective presentation:
$$
\Hom_{\underline \calC}(-,\tau_\calC C_2\oplus I)\xrightarrow{(-,\underline q)} \Hom_{\underline \calC}(-, M)\to F\to 0.
$$
We claim that it is a minimal projective presentation. 
Otherwise, by Yoneda's Lemma, $q$ has a direct summand of the form $X=X$ or $X\to 0$. Since $p$ is indecomposable in $\calS(\calC)_\calI$,  $q$ is indecomposable in $\calF(\calC)_\calI$ due to the RSS-equivalence (Proposition \ref{prop:RSS}). Therefore $q$ is either a morphism $X=X$ or $X\to 0$. But this implies $F=0$, which contradicts with the fact that $F$ is the starting term of an almost split sequence.

On the other hand, $F$ also has a minimal projective presentation:
$$
\Hom_{\underline \calC}(-,A_1)\xrightarrow{(-,\underline f)} \Hom_{\underline \calC}(-,A_2)\to F\to 0.
$$

By the Comparison Lemma, $\underline f\cong \underline q$ in $\calH(\underline\calC)$.
By the dual of Theorem \ref{thm:epivalence}, there is an epivalence $\pi':\underline{\calF(\calC)}\to \calH(\underline\calC)$, which reflects isomorphisms. So we have $f\cong q=\cok\mimo\tau_\calC {\underline h} \cong\bcok\Mimo\tau_\calC \underline h$ in $\underline{\calF(\calC)}$, which finishes the proof of Theorem~\ref{thm:A}.   

\section{Monomorphism categories of Frobenius categories}
\subsection{} In this section, we will study the monomorphism categories of Frobenuis categories and their Auslander-Reiten translations. The following result is well-known.

\begin{lem}\cite[Lemma 2.1]{C1}\cite[Corollary 5.2]{ZX}\label{lem:Chen}
Let $\calC$ be an extension closed subcategory of an abelian category. Then $\calS(\calC)$ is a
 Frobenius category if and only if so is $\calC$.  
\end{lem}

Hence, for a Frobenius category $\calC$, both $\overline\calC$ and $\overline{\calS(\calC)}$ are triangulated categories. First, we shall compare their suspension functors. Denote by $[1]$ and $\la 1\ra$ for the suspension functors on $\overline\calC$ and $\overline{\calS(\calC)}$ respectively.
\begin{prop} \label{prop:shift}
For any object $f\in \overline{\calS(\calC)}$, $f\la 1\ra\cong \Mimo (\overline f[1])$.
\end{prop}
\begin{proof}
It suffices to prove for indecomposable objects $\left(\begin{smallmatrix}A\\B\end{smallmatrix}\right)_f\in\calS(\calC)_\calI$. Suppose $i_B:B\to I(B)$ is an $\Ext$-injective envelope and $i_A=i_B\circ f$. Then there is a commutative diagram with exact rows:
$$
\xymatrix{0\ar[r]&A\ar[r]^{i_A}\ar[d]^f&I(B)\ar[r]^{\pi_A}\ar@{=}[d]&A'\ar[r]\ar[d]^{f'}&0\\
0\ar[r]&B\ar[r]^{i_B}&I(B)\ar[r]^{\pi_B}&B'\ar[r]&0,\\}
$$
where  $f'$ is an epimorphism satisfying $\overline{f'}\cong \overline f[1]$ in $\overline\calC$. Let $\Mimo f'=[f',e]: A'\to B'\oplus I'$. Then it is easy to check that there is an exact sequence in $\calS(\calC)$ below:
$$
\xymatrix{0\ar[r]&A\ar[r]^{i_A}\ar[d]^f&I(B)\ar[r]^{\pi_A}\ar[d]^{\left[\begin{smallmatrix}1\\0\end{smallmatrix}\right]}&A'
\ar[r]\ar[d]^{\left[\begin{smallmatrix}f'\\e\end{smallmatrix}\right]}&0\\
0\ar[r]&B\ar[r]^{\left[\begin{smallmatrix}i_B\\0\end{smallmatrix}\right]}&I(B)\oplus I'\ar@< 2pt>[r]^{\left[\begin{smallmatrix}\pi_B &0\\ e\pi_A&1\end{smallmatrix}\right]}&B'\oplus I'\ar[r]&0.\\}
$$
Hence, in the stable category $\overline{\calS(\calC)}$, we have $f\la 1\ra\cong \Mimo f'\cong \Mimo (\overline f[1])$ by the Remark \ref{rem:mimo_well_def}.
\end{proof}

Throughout the rest of this section,   let $\calV$ be a dualizing $k$-variety and $\calC$ be a functorially finite Frobenius subcategory of $\calV\modd$ having enough $\Ext$-projective and $\Ext$-injective  objects. Denote by $\tau$ the Auslander-Reiten translation on $\calC$. According to Theorem \ref{thm:A} and Lemma \ref{lem:Chen}, $\calS(\calC)$ is again a Frobenius category having Auslander-Reiten-Serre duality. Denote by $\tau_\calS$ the Auslander-Reiten translation on $\calS(\calC)$. Notice that triangle functors $\tau$ and $\tau_\calS$ yield auto-equivalences of the stable categories $\overline\calC$ and $\overline{\calS(\calC)}$ respectively.

To deduce a formula for the triangle functor $\tau_\calS$, we need the following notation.

\begin{defn}
Let $\left(\begin{smallmatrix}X\\Y\end{smallmatrix}\right)_{\overline f}\in\calH(\overline\calC)$, which fits into a distinguished triangle $X\stackrel{\overline f}\to Y\stackrel{\overline g}\to Z\stackrel{\overline h}\to X[1]$ in $\overline\calC$, the rotation operation $\Rot \left(\begin{smallmatrix}X\\Y\end{smallmatrix}\right)_{\overline f}$ is defined to be $\left(\begin{smallmatrix}Y\\Z\end{smallmatrix}\right)_{\overline g}$.
\end{defn}
We will write $\Rot(\overline f)$ for convenience. From the axioms of triangulated category, it is clear that $\Rot$ is well-defined on objects of $\calH(\overline\calC)$.  However, we warn the reader that in general it is not a functor.

\begin{prop} \ \label{Rot-Cok}
\begin{enumerate}
\item For any $\overline f\in{\calH(\overline\calC)}$, $\Rot^3 \overline f \cong -\overline f[1]$ and $(\Rot\overline f)[1]=\Rot (\overline f[1])$.
\item For any $f\in\overline{\calS(\calC)}$, $\overline{\bcok f}=\Rot \overline f$.
\end{enumerate}
\end{prop}
\begin{proof}\ The formula  $\Rot^3 \overline f \cong -\overline f[1]$ follows from the axioms of triangulated category  
directly. For the formula  $\overline{\bcok f}=\Rot \overline f$, its proof is essentially  the same as Lemma 6.1 in \cite{RS2}. Here we only need to prove $(\Rot\overline f)[1]=\Rot (\overline f[1])$. Let  $X\stackrel{\overline f}\to Y\stackrel{\overline g}\to Z\stackrel{\overline h}\to X[1]$ be a distinguished triangle in $\overline\calC$. Then by the axioms of triangulated category, 
   $X[1]\stackrel{-\overline {f}[1]}\longrightarrow Y[1]\stackrel{-\overline {g}[1]} \longrightarrow Z[1]\stackrel{-\overline {h}[1]} \longrightarrow  X[2]$ is a distinguished triangle, which is isomorphic to the triangle $X[1]\stackrel{\overline {f}[1]}\longrightarrow Y[1]\stackrel{\overline {g}[1]}\longrightarrow Z[1]\stackrel{-\overline {h}[1]}\longrightarrow X[2]$. Hence by the definition of $\Rot$, we have  $(\Rot\overline f)[1]=\Rot (\overline f[1])$. 
\end{proof}

Recall that the stable quotient $\pi:\overline{\calS(\calC)}\to\calH(\overline\calC)$ is a full and dense functor.
\begin{lem}\label{lem:rot}
Let $ f\in\overline{\calS(\calC)}$. Then $\overline{\tau_\calS f}\cong \tau\Rot\overline f$ in $\calH(\overline{\calC})$.
\end{lem}
\begin{proof}
By Theorem \ref{thm:A} and Proposition \ref{Rot-Cok} (2), $\overline{\tau_\calS f}\cong \overline{\Mimo\tau\underline{\bcok f}}\cong \tau\overline{\bcok f} \cong \tau\Rot\overline f$.  
\end{proof}

\subsection{Stably Calabi-Yau Frobenius categories}
Recall that on a $k$-linear $\Hom$-finite additive category $\calA$, a right Serre functor is an additive functor $F:\calA\to\calA$ together with isomorphisms: $$\eta_{A,B}:\Hom(A,B)\to D\Hom(B,FA)$$ for any $A,B\in\calA$ and natural in $A$ and $B$. A left Serre functor $G$ is an additive functor $G:\calA\to \calA$ together with natural isomorphisms $$\xi_{A,B}:\Hom(A,B)\to D\Hom(GB,A).$$
Any left or right Serre functor is full-faithful \cite{RV}. A Serre functor is a right Serre functor $\bS$ which is dense. It follows that $\calA$ admits a Serre functor if and only if it has both a left and a right Serre functor \cite{RV}.

A $k$-linear $\Hom$-finite triangulated category $\calT$ is called $d$-Calabi-Yau, if $\calT$ has a Serre functor $\bS=[d]$, i.e. there is a bifunctorial isomorphisms $\Hom(A,B)\cong D\Hom(B,A[d])$ for all $A,B\in\calT$. Due to a well-known result from \cite{RV}, for a $k$-linear $\Hom$-finite triangulated category $\calT$, $\calT$ has Auslander-Reiten triangles if and only if it admits a Serre functor.

A $k$-linear $\Hom$-finite Frobenius category $\calC$ is called stably $d$-Calabi-Yau, if the stable category $\underline \calC$ is a $d$-Calabi-Yau triangulated category. Furthermore for a stably $d$-Calabi-Yau Frobenius category, there is a functorial isomorphism $\tau\cong[d-1]$, where $\tau:\underline \calC\to\underline \calC$ is the Auslander-Reiten translation.

Applying Theorem \ref{thm:A} for the monomorphism category $\calS(\calC)$, we have the following computation: 
     \begin{align}
     \tau_\calS f &\cong\Mimo\overline{\tau_\calS f} \tag{Lemma~\ref{lem:mimo_stable}}\\ 
     &\cong\Mimo\tau\Rot\overline f . \tag{Lemma~\ref{lem:rot}} \end{align}

Therefore $\tau_\calS^2 f \cong \Mimo\tau\Rot\overline {\Mimo\tau\Rot\overline f}  
      \cong\Mimo\tau\Rot\tau\Rot\overline f$.
      
So by induction, 
 \begin{align*}
 \tau^6_\calS f&\cong\Mimo (\tau\Rot)^6\overline f\\
 &\cong\Mimo ([d-1]\Rot)^6\overline f\\
 &\cong\Mimo\Rot^6(\overline f[6d-6])\\
 &\cong\Mimo \overline f[6d-4]\\
 &\cong f\la 6d-4\ra.&  \tag{Prop.~\ref{prop:shift}}
 \end{align*}
This proves Theorem \ref{thm:B}.

\begin{rem}\label{rem:fCY}
Theorem  \ref{thm:B} suggests that the monomorphism category $\calS(\calC)$ of a stably $d$-Calabi-Yau Frobenius category $\calC$  has likely a fractional Calabi-Yau-dimension. Indeed, we can also compute $\tau_\calS f$~as:
 \begin{align*}
     \tau_\calS f &\cong\Mimo\overline{\tau_\calS f} \tag{Lemma~\ref{lem:mimo_stable}}\\ 
     &\cong\Mimo\tau\overline{\bcok f}  \\
     &\cong\Mimo(\overline{\bcok f}[d-1]) \\
      &\cong(\Mimo \overline{\bcok f})\la d-1\ra \tag{Prop.\ref{prop:shift}}\\ 
     &\cong (\Mimo\bcok f)\la d-1\ra \end{align*} 
Notice that $\Mimo\bcok\la d-1\ra$ is a functor $\overline{\calS(\calC)}\to \overline{\calS(\calC)}$. Although on many examples we can verify that $\Mimo\bcok\la d\ra$ is the Serre functor, it is still an open question for us to prove the functorial isomorphism $\tau_\calS \cong \Mimo\bcok\la d-1\ra$.  
\end{rem}

\section{Examples}
\subsection{The monomorphism power categories}

Let $\calC$ be an exact category. Inductively, set $\calS^0(\calC)=\calC$ and for a natural number $n\geq 1$ define the {\bf monomorphism power category} $\calS^n(\calC):=\calS(\calS^{n-1}(\calC))$. As an immediate consequence of Lemma \ref{lem:Chen}, we have the following:

\begin{prop}
Let $\calC$ be a Frobenius subcategory of $\m\modd$. Then the monomorphism category $\calS^n(\calC)$ is a Frobenius category.
\end{prop}

\begin{exm}
Let $A=k[x]/\langle x^2\rangle$. We compute the Auslander-Reiten quiver for $\calS^2(A\modd)$.
\end{exm}

 Let $S$ be the unique simple module of $A$. By the construction of  $\calS^2(A\modd)$, we know that $\calS^2(A\modd)=\{ \rm{inflations \ in \ }\calS(A\modd)\}$. For the inflation $\phi: \left(\begin{smallmatrix}X\\Y\end{smallmatrix}\right)_{ f}\to \left(\begin{smallmatrix}X'\\Y'\end{smallmatrix}\right)_{ f'}$ which is indecomposable in $\calS^2(A\modd)$, we simply write it as
$\begin{matrix}{\begin{smallmatrix}X\\Y\end{smallmatrix}} \\ {\begin{smallmatrix}X'\\Y'\end{smallmatrix}}
\end{matrix}$, if $\phi$ is uniquely determined by $X, Y, X', Y'$. By theorem $A$, We can compute the Auslander-Reiten quiver for $\calS^2(A\modd)$ as follows:
 
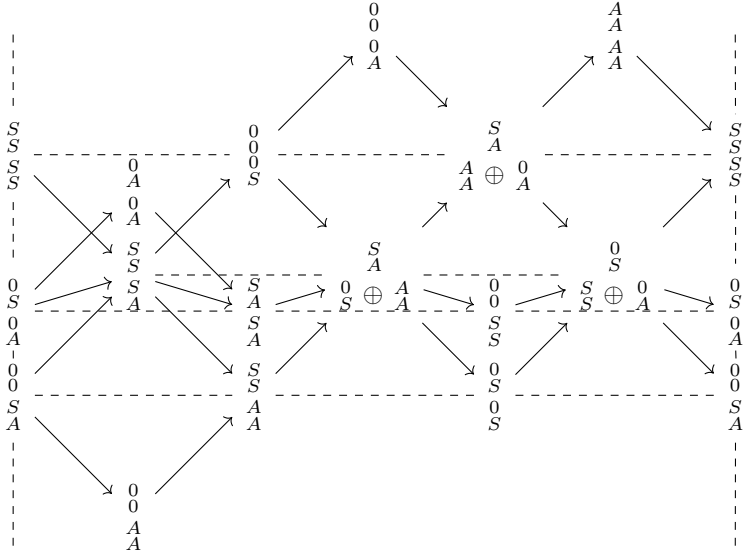
\begin{figure}[h]
\begin{tikzpicture}[line width=0.4pt,scale=1.6]
% \draw[help lines] grid(20,10);
\node (a1) at (2,1.7) {$ {\begin{matrix}{\begin{smallmatrix}0\\S\end{smallmatrix}} \\ {\begin{smallmatrix}0\\A\end{smallmatrix}}
\end{matrix}}$};
\node (a2) at (3,2) {$ {\begin{matrix}{\begin{smallmatrix}S\\S\end{smallmatrix}} \\ {\begin{smallmatrix}S\\A\end{smallmatrix}}
\end{matrix}}$};
\node (a3) at (4,1.7) {$ {\begin{matrix}{\begin{smallmatrix}S\\A\end{smallmatrix}} \\ {\begin{smallmatrix}S\\A\end{smallmatrix}}
\end{matrix}}$};
\node (a4) at (5,2) {$\begin{matrix}{\begin{smallmatrix}S\\A\end{smallmatrix}} \\ {{\begin{matrix}{\begin{smallmatrix}0\\S\end{smallmatrix}} \oplus {\begin{smallmatrix}A\\A\end{smallmatrix}}
\end{matrix}} }
\end{matrix}$};
\node (a5) at (6,1.7) {$ {\begin{matrix}{\begin{smallmatrix}0\\0\end{smallmatrix}} \\ {\begin{smallmatrix}S\\S\end{smallmatrix}}
\end{matrix}}$};

\node (a6) at (7,2) {$\begin{matrix}{\begin{smallmatrix}0\\S\end{smallmatrix}} \\ {{\begin{matrix}{\begin{smallmatrix}S\\S\end{smallmatrix}} \oplus {\begin{smallmatrix}0\\A\end{smallmatrix}}
\end{matrix}} }
\end{matrix}$};

\node (a7) at (8,1.7) {$ {\begin{matrix}{\begin{smallmatrix}0\\S\end{smallmatrix}} \\ {\begin{smallmatrix}0\\A\end{smallmatrix}}
\end{matrix}}$};

\node (b1) at (2,3) {$ {\begin{matrix}{\begin{smallmatrix}S\\S\end{smallmatrix}} \\ {\begin{smallmatrix}S\\S\end{smallmatrix}}
\end{matrix}}$};
\node (b2) at (3,2.7) {$ {\begin{matrix}{\begin{smallmatrix}0\\A\end{smallmatrix}} \\ {\begin{smallmatrix}0\\A\end{smallmatrix}}
\end{matrix}}$};
\node (b3) at (4,3){$\begin{smallmatrix}{\begin{smallmatrix}0\\0\end{smallmatrix}} \\ {\begin{smallmatrix}0\\S\end{smallmatrix}}\end{smallmatrix}$};

\node (b4) at (6,3) {$\begin{matrix}{\begin{smallmatrix}S\\A\end{smallmatrix}} \\ {{\begin{matrix}{\begin{smallmatrix}A\\A\end{smallmatrix}} \oplus {\begin{smallmatrix}0\\A\end{smallmatrix}}
\end{matrix}} }
\end{matrix}$};
\node (b5) at (8,3){$\begin{smallmatrix}\begin{smallmatrix}S\\S\end{smallmatrix} \\ \begin{smallmatrix}S\\S\end{smallmatrix}
\end{smallmatrix}$};

\node (c1) at (2,1) {$ {\begin{matrix}{\begin{smallmatrix}0\\0\end{smallmatrix}} \\ {\begin{smallmatrix}S\\A\end{smallmatrix}}
\end{matrix}}$};
\node (c2) at (4,1) {$ {\begin{matrix}{\begin{smallmatrix}S\\S\end{smallmatrix}} \\ {\begin{smallmatrix}A\\A\end{smallmatrix}}
\end{matrix}}$};
\node (c3) at (6,1) {$ {\begin{matrix}{\begin{smallmatrix}0\\S\end{smallmatrix}} \\ {\begin{smallmatrix}0\\S\end{smallmatrix}}
\end{matrix}}$};
\node (c4) at (8,1) {$ {\begin{matrix}{\begin{smallmatrix}0\\0\end{smallmatrix}} \\ {\begin{smallmatrix}S\\A\end{smallmatrix}}
\end{matrix}}$};

\node (d1) at (5,4) {$ {\begin{matrix}{\begin{smallmatrix}0\\0\end{smallmatrix}} \\ {\begin{smallmatrix}0\\A\end{smallmatrix}}
\end{matrix}}$}; 
\node (d2) at (7,4) {$ {\begin{matrix}{\begin{smallmatrix}A\\A\end{smallmatrix}} \\ {\begin{smallmatrix}A\\A\end{smallmatrix}}
\end{matrix}}$};
 
\node (e1) at (3,0) {$ {\begin{matrix}{\begin{smallmatrix}0\\0\end{smallmatrix}} \\ {\begin{smallmatrix}A\\A\end{smallmatrix}}
\end{matrix}}$};

\draw[->] (a1)--(a2);\draw[->] (a1)--(b2);
\draw[->] (a2)--(a3);\draw[->] (a2)--(c2);\draw[->] (a2)--(b3);\draw[->] (a3)--(a4);\draw[->] (a4)--(a5);\draw[->] (a4)--(b4);\draw[->] (a4)--(c3);\draw[->] (a5)--(a6);\draw[->] (a6)--(a7);
\draw[->] (a6)--(b5);\draw[->] (a6)--(c4);

 \draw[->] (b1)--(a2);\draw[->] (b2)--(a3); \draw[->] (b3)--(a4);\draw[->] (b3)--(d1); \draw[->] (b4)--(a6); \draw[->] (b4)--(d2); 
 \draw[->] (d1)--(b4); \draw[->] (d2)--(b5); 
\draw[->] (c1)--(a2);\draw[->] (c1)--(e1);\draw[->] (c2)--(a4);\draw[->] (c3)--(a6);

\draw[->] (e1)--(c2); 
\draw[dashed] (a1)--(a3)--(a5)--(a7); 
\draw[dashed] (a2)--(a4)--(a6); 
\draw[dashed] (b1)--(b3)--(b4)--(b5); 
\draw[dashed] (c1)--(c2)--(c3)--(c4);  
\draw[dashed] (2,4)--(b1)--(a1)--(c1)--(2,-0.3);
\draw[dashed] (8,4)--(b5)--(a7)--(c4)--(8,-0.3);
\end{tikzpicture}

\end{figure}

\newpage

 A finite dimensional Iwanaga-Gorenstein algebra $A$ is called stably $d$-Calabi-Yau, if the category of Gorenstein projective $A$-modules is a stably $d$-Calabi-Yau Frobenius category. A finite dimensional algebra $A$ is called $d$-Calabi-Yau tilted if it is isomorphic to the endomorphism algebra of some $d$-cluster tilting object in a $d$-Calabi-Yau triangulated category. Keller and Reiten showed that if $A$ is a 2-Calabi-Yau tilted algebras, then $A$ is 1-Iwanaga-Gorenstein and stably 3-Calabi-Yau \cite{KR}. Hence the subcategory of Gorenstein projective $A$-modules $\Gp(A)$ is a functorially finite and stably 3-Calabi-Yau Frobenius subcategory. Let $\Lambda=\begin{bmatrix} A&A\\0&A
 \end{bmatrix}$ be the upper triangular matrix algebra. Since $\m\modd\cong\calH(A)$, every $\Lambda$-module is an $A$-homomorphism $\left(\begin{smallmatrix}X\\Y
 \end{smallmatrix}\right)_f$ and $\Gp(\Lambda)\cong \calS(\Gp(A))$ (Theorem \ref{thm:t2a}).  
By Remark \ref{rem:fCY}, 
$
\tau_{Gp(\m)}\left(\begin{smallmatrix}X\\Y
 \end{smallmatrix}\right)_f \cong \Mimo\bcok \left(\begin{smallmatrix}X\\Y
 \end{smallmatrix}\right)_f \la2\ra.
$

\begin{exm}
Let $A$ be the following bounded quiver algebra with relation $\beta\delta, \delta\gamma, \gamma\beta$.
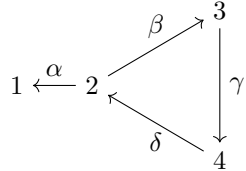
$$
\begin{tikzpicture}
\node (A) at (0,0) {$1$}; 
  \node(B) at (1,0) {$2$};
  \node (C) at (2.7,1) {$3$}; 
  \node (D) at (2.7,-1) {$4$};
  \draw[->] (B)--node[above]{$\alpha$}(A);
  \draw[->] (B)--node[above]{$\beta$}(C);
  \draw[->]  (C)--node[right]{$\gamma$}(D);
  \draw[->]  (D)--node[below]{$\delta$}(B);
\end{tikzpicture}
$$
The algebra $A$ is cluster-tilted, and hence  $2$-Calabi–Yau tilted. Therefore $\Gp(A)$ is a $3$-Calabi-Yau Frobenius category.
\begin{figure}[h]
\begin{tikzpicture}
% \draw[help lines] grid(20,10);
\node[red] (a1) at (2,5) {$4$};
\node[red] (a2) at (4,5) {$3$};
\node[red] (a3) at (6,5) {$rP_4$};
\node (a4) at (8,5) {$I_2$};
\node[red] (b1) at (5,4) {$P_2$};
\node (b2) at (7,4) {$2$};
\node[red] (b3) at (9,4) {$4$};
\node[red] (p1) at (3,6) {$P_3$};
\node[red] (p2) at (7,6) {$P_4$};
\node[red] (q1) at (4,3) {$1$};
\node (q2) at (6,3) {$I_3$};
\draw[->] (a2)--(b1);\draw[->] (a3)--(b2);\draw[->] (a4)--(b3);
\draw[->] (b1)--(a3);\draw[->] (b2)--(a4); 
\draw[->] (a1)--(p1);\draw[->] (a3)--(p2);\draw[->] (p1)--(a2);\draw[->] (p2)--(a4);
\draw[->] (q1)--(b1);\draw[->] (b1)--(q2); \draw[->] (q2)--(b2);
\draw[dashed] (a1)--(a2)--(a3)--(a4);
\draw[dashed] (b1)--(b2)--(b3);\draw[dashed] (q1)--(q2);
\draw[dashed] (2,6)--(a1)--(2,3);\draw[dashed] (9,6)--(b3)--(9,3);
\end{tikzpicture}
\caption{Auslander-Reiten quiver of $A\modd$
}\label{fig:2cy}
\end{figure}
In Figure \ref{fig:2cy},
$P_2=\begin{smallmatrix}&2\\1&&3\end{smallmatrix}$, 
$P_3=\begin{smallmatrix}3\\4\end{smallmatrix}$,
$P_4=\begin{smallmatrix}4\\2\\1\end{smallmatrix}$,
$rP_4=\begin{smallmatrix}2\\1\end{smallmatrix}$,
$I_2=\begin{smallmatrix}4\\2\end{smallmatrix}$,
$I_3=\begin{smallmatrix}2\\3\end{smallmatrix}$ and $\Gp(A)$ are marked red.
The monomorphism category $\calS(\Gp(A))$ contains $17$ indecomposable objects.

\begin{figure}[h]
\begin{tikzpicture}
%\draw[help lines] grid(20,10);
\node (a1) at (2,5) {$\begin{smallmatrix}3\\P_2\end{smallmatrix}$};
\node (a2) at (4,5) {$\begin{smallmatrix}rP_4\\rP_4\end{smallmatrix}$};
\node (a3) at (6,5) {$\begin{smallmatrix}0\\4\end{smallmatrix}$};
\node (a4) at (8,5) {$\begin{smallmatrix}4\\P_3\end{smallmatrix}$};
\node (a5) at (10,5) {$\begin{smallmatrix}3\\3\end{smallmatrix}$};
\node (b1) at (3,4) {$\begin{smallmatrix}0\\rP_4\end{smallmatrix}$};
\node (b2) at (5,4) {$\begin{smallmatrix}rP_4\\P_4\end{smallmatrix}$};
\node (b3) at (7,4) {$\begin{smallmatrix}4\\4\end{smallmatrix}$};
\node (b4) at (9,4) {$\begin{smallmatrix}0\\3\end{smallmatrix}$};
\node (b5) at (11,4) {$\begin{smallmatrix}3\\P_2\end{smallmatrix}$};
\node (p1) at (3,6) {$\begin{smallmatrix}P_2\\P_2\end{smallmatrix}$};
\node (p2) at (7,6) {$\begin{smallmatrix}0\\P_3\end{smallmatrix}$};
\node (p3) at (9,6) {$\begin{smallmatrix}P_3\\P_3\end{smallmatrix}$};
\node (q1) at (4,3) {$\begin{smallmatrix}0\\P_4\end{smallmatrix}$};
\node (q2) at (6,3) {$\begin{smallmatrix}P_4\\P_4\end{smallmatrix}$};
\node (q3) at (10,3) {$\begin{smallmatrix}0\\P_2\end{smallmatrix}$};
\node(t1) at (11,2)  {$\begin{smallmatrix}1\\1\end{smallmatrix}$};
\draw[->] (a1)--(b1);\draw[->] (a2)--(b2);  \draw[->] (a3)--(b3);  \draw[->] (a4)--(b4);  \draw[->] (a5)--(b5);   
\draw[->](b1)--(a2);\draw[->](b2)--(a3);\draw[->](b3)--(a4);\draw[->](b4)--(a5);
\draw[->] (b1)--(q1);\draw[->] (b2)--(q2);\draw[->] (b4)--(q3);
\draw[->] (q1)--(b2);\draw[->] (q2)--(b3);\draw[->] (q3)--(b5);
\draw[->] (a1)--(p1);\draw[->] (a3)--(p2);\draw[->] (a4)--(p3);
\draw[->] (p1)--(a2);\draw[->] (p2)--(a4);\draw[->] (p3)--(a5);
\draw[dashed] (a1)--(a2)--(a3)--(a4)--(a5)--(11,5);
\draw[dashed] (2,4)--(b1)--(b2)--(b3)--(b4)--(b5);
\node (d1) at (2,7) {$\begin{smallmatrix}1\\1\end{smallmatrix}$};
\node (d2) at (9,2) {$\begin{smallmatrix}0\\1\end{smallmatrix}$};
\draw[->] (d1)--(p1); \draw[->] (d2)--(q3); \draw[->] (d2)--(t1);
\draw[dashed] (2,6)--(a1)--(2,3);
\draw[dashed] (11,6)--(b5)--(11,3);
\end{tikzpicture}
\caption{Auslander-Reiten quiver of $\calS(\Gp(A))$
}
\end{figure}
\end{exm}

 \newpage

{\bf Acknowledgments:} We would like to thank  Rasool Hafezi, Sondre Kvamme, Claus Michael Ringel and Markus Schmidmeier for helpful discussions and comments.

%%%%%%%%%%%%%%%

%%%

\end{document}